\documentclass[a4paper,11pt]{amsart}
\usepackage{fullpage}
\usepackage{tikz}
\usetikzlibrary{matrix,arrows,decorations.pathmorphing}
\usepackage{amsmath}
\numberwithin{equation}{section}
\usepackage{amssymb}
\usepackage{bbm}
\usepackage{comment}
\usepackage{amsthm}
\usepackage[all,cmtip]{xy}
\usepackage[toc,page]{appendix}

\usepackage[british]{babel}
\usepackage{courier}
\usepackage{enumerate}
\usepackage[T1]{fontenc}
\usepackage{listings}
\usepackage[prependcaption,textsize=tiny,textwidth=2cm]{todonotes}
\theoremstyle{plain}

\newtheorem{theorem}{Theorem}[section]
\newtheorem*{theorem*}{Theorem}

\newtheorem{lemma}[theorem]{Lemma}
\newtheorem{proposition}[theorem]{Proposition}
\newtheorem{cor}[theorem]{Corollary}
\theoremstyle{remark}

\newtheorem{remark}[theorem]{Remark}

\theoremstyle{definition}
\newtheorem*{ack}{Acknowledgments}
\newcommand{\id}{\mathbbm{1}}
\newcommand{\map}{\mathrm{Map}}
\newcommand{\mapp}{\mathrm{Map}_*}
\newcommand{\mat}{\mathrm{Mat}}
\newcommand{\Z}{\mathbb{Z}}
\title{On gauge groups over high dimensional manifolds\\ and self-equivalences of $H$-spaces}
\author{Ingrid Membrillo-Solis}
\address{Mathematical Sciences, University of Southampton, University Road, Southampton SO17~1BJ, United Kingdom}
\email{i.membrillo-solis@soton.ac.uk}
\date{\today}
\subjclass[2010]{55P45, 54C35, 55P10 (primary), 55R10, 55R25 (secondary)}
\keywords{Gauge groups, sphere bundles, self-equivalences, homotopy decomposition}

\begin{document}

\begin{abstract}
Let $Y$ be a pointed space and let $\mathcal E(Y^r)$ be the group of based self-equivalences of $Y^r$, $r\geq 2$. For $Y$ a homotopy commutative $H$-group we construct a subgroup $\mathcal E_{\mat}(Y^r)$ of $\mathcal E(Y^r)$ which has a group structure isomorphic to either $GL_r(\mathbb Z)$, or $GL_r(\mathbb Z_d)$, $d\geq 2$.  We classify principal bundles over connected sums of $q$-sphere bundles over $n$-spheres, and use the group $\mathcal E_{\mat}(Y^r)$ to obtain homotopy decompositions of their gauge groups. Using these decompositions we give an integral classification, up to homotopy, of the gauge groups of principal $SU(2)$-bundles over certain 2-connected 7-manifolds with torsion-free homology. 
\end {abstract}

\maketitle

\section{Introduction}
Let $X$ be a connected $CW$-complex and $G$ be a Lie group. The isomorphism classes of principal $G$-bundles over $X$ are in one-to-one correspondence with the set $[X,BG]$ of unpointed homotopy classes of maps from $X$ to the classifying space of $G$. Given a principal $G$-bundle $P$ over $X$, the gauge  group of $P$, denoted $\mathcal G^P(X)$, is the group of its bundle  automorphisms covering the identity on $X$. The pointed gauge group $\mathcal G_*^P(X)$ consists of all bundle automorphisms that pointwise fix the fibre at the base point. Endowed with the compact-open topology, $\mathcal G^P(X)$ (resp.\ $\mathcal G_*^P(X))$ is homotopy equivalent to the loop space of the connected component of the mapping space ${\rm{Map}}(X,BG)$ (resp.\ ${\rm{Map}}_*(X,BG))$ which contains the map that classifies the bundle \cite{AB}. Although the set of isomorphism classes of principal $G$-bundles over a finite $CW$-complex $X$ might be infinite, there exist only finitely many distinct homotopy types among the gauge groups \cite{CS}.

The topology of gauge groups and their associated classifying spaces  has received considerable attention due to their connections to mathematical physics and other areas in mathematics. Particular attention has been paid in counting the number of homotopy types of gauge groups over surfaces and 4-manifolds. Although new ideas coming from differential geometry and mathematical physics suggest the possibility of extending gauge theories to high dimensions \cite{DT}, the homotopy theory of  gauge groups over high dimensional manifolds has been barely explored. Research done in this direction includes the study of gauge groups over high dimensional spheres \cite{HKK,KKT,Spf}, $S^3$-bundles over $S^4$ \cite{IMS} and $(n-1)$-connected $2n$-manifolds \cite{RH1}. 

The study of the group of self-homotopy equivalences $\mathcal E(Y)$ of a topological space $Y$ has a long tradition in homotopy theory (see for instance \cite{Ark},\cite{Rut2}). Very little is known, however, about this group when $Y$ is a product or a wedge sum of spaces. Moreover, the applications of the group of self-equivalences require further investigations. In this work  we present an application of the group $\mathcal E(Y)$ to the homotopy classification of mapping spaces. 

It is well known that there are homeomorphisms $\map (\bigsqcup_{i=1}^r X_i,Z)\cong  \prod_{i=1}^r \map (X_i,Z)$ and $\mapp(\bigvee_{i=1}^r X_i,Z)\cong  \prod_{i=1}^r \mapp (X_i,Z)$. Similar decompositions do not hold in general for $\map(\bigvee_{i=1}^r X_i,Z)$. It has been shown that the loops on these kinds of mapping spaces appear in homotopy decompositions of gauge groups of principal $G$-bundles over  7-dimensional manifolds \cite{IAM2}. For suitable spaces $X_i$ and $Z$, $Y=\mapp(X_i,Z)$ is a homotopy commutative  $H$-group and so is $\prod_{i=1}^r\mapp(X_i,Z)$. We construct a subgroup of the group of self-equivalences of the $r$-fold cartesian product of certain homotopy commutative $H$-groups $Y$. We show that this construction can be used to decompose the  loop spaces of the connected components of $\map(\bigvee_{i=1}^r X_i,Z)$. As a result, we provide full decompositions of the gauge groups associated to $q$-sphere bundles over $n$-spheres. Finally, we use these decompositions to give a complete classification,  up to homotopy, of $SU(2)$-gauge groups over connected sums of certain $S^3$-bundles over $S^4$ with cross sections.

Let $\mat_r(\Z)$ be the set of $r$-by-$r$ matrices with integer coefficients, and let $GL_r(\Z)$ be the subset of invertible matrices.  Given an $H$-group $Y$, and a matrix $A = (a_{ij}) \in \mat_r(\Z)$, we define a self-map, called a matrix map, on the $r$-fold cartesian product of $Y$ as the composite
\begin{equation}
\xymatrix{
\prod_{j=1}^rY\ar[r]^-{\Delta^r}&\prod_{i=1}^r\prod_{j=1}^rY\ar[rr]^-{\prod_{i=1}^r\prod_{j=1}^r \mathbbm a_{ij}}&&\prod_{i=1}^r \prod_{j=1}^rY\ar[rr]^-{\prod_{i=1}^r\mu^r}&&\prod _{i=1}^rY,
}
\end{equation}
where $\Delta^r$ is the $r$-fold diagonal map, $\mu^r$ is the $r$-fold multiplication and $\mathbbm{a}_{ij}$ is the $a_{ij}$-th power map. We use the notation $\mathcal E_A$ for the matrix map associated to $A\in GL_r(\Z)$. The greatest common divisor of $x_1,\dots,x_r\in\mathbb Z_m$ is defined as ${\gcd}_m(x_1,\dots,x_r):=\gcd(\tilde x_1,\dots,\tilde x_r,m)$, where $\tilde x_i$ is any representative of the class of $x_i$. Let $\mathbb Z_0:=\mathbb Z$ and $\mathbbm 1_Y:Y\to Y$ be the identity map. Given a group $G$ and an element $g\in G$, let $o(g)$ be the order of $g$. Our first result is a characterisation of $\mathcal E_{\mat}(Y^r)$, the set of homotopy classes of matrix maps $\mathcal E_A$ of $Y^r$. Given a space $X$ such that $\pi_0(X)=\mathbb Z_m^r$, let $X_K$ be the connected component of $X$ indexed by $K=(k_1,\dots,k_r)\in \mathbb Z_m^r$.

\begin{theorem}\label{t:GLn}
Let $Y$ be an $H$-group and let $r \geq 2$. If $Y$ is homotopy commutative then the following hold: 
\begin{enumerate}
\item The homotopy set $\mathcal E_{\mat}(Y^r)$ has a group structure isomorphic to $GL_r(\mathbb Z)$, if $o(\mathbbm1_Y)$ is infinite, or $GL_r(\mathbb Z_d)$, if $o(\mathbbm 1_Y)=d\geq2$;
\item Every map $\mathcal E_A\in\mathcal E_{\mat}(Y^r)$ is an $H$-map;
\item  Suppose $\pi_0(Y)=\mathbb Z_m$ for some $m \geq 0$. Let $K=(k_1,\dots,k_r)$ and $K'=(k_1',\dots,k_r')$, where $k_i,k_i'\in\pi_0(Y)$, for all $1\leq i\leq r$. Then there exists a matrix $A\in GL_r(\mathbb Z)$ such that the restriction of the map $\mathcal E_A$ to $(Y^r)_K$ induces a self-equivalence $(Y^r)_K \simeq (Y^r)_{K'}$ if and only if $\gcd_m(k_1,\dots,k_r)=\gcd_m(k_1',\dots,k_r')$. 
\end{enumerate}
\end{theorem}
Theorem \ref{t:GLn} shows that if the order of the identity map of a homotopy commutative $H$-group $Y$  is infinite, then there are infinitely many classes of self-$H$-equivalences  of the $r$-fold cartesian product of $Y$. Examples of homotopy commutative $H$-groups include Eilenberg-MacLane spaces $K(\Gamma, n)$, $n\geq1$, where $\Gamma$ is an abelian group and, more generally, the loop spaces of $H$-groups. 
In \cite{Saw} Sawashita describes the group of self-$H$-equivalences of products $X_1\times\cdots\times X_r$ of $H$-groups $X_i$. We construct an explicit subgroup of this group for the case when the $H$-groups $X_i$ coincide and are homotopy commutative. 

We also explore the action of the group of self equivalences $\mathcal E_{\mat}(Y^r)$ on homotopy sets. 
Given a map $f:X\to Y$, for $1 \leq i \leq n$ let $f_i:X\to \prod _{i=1}^r Y$ be the map such that $$p_j\circ f_i=\begin{cases} f & j=i\\ 0 &j\neq i\end{cases}$$ where $p_j$ is the projection onto the the $j$-th factor.  For $Y$ an $H$-group and $k\in\mathbb Z$, let $F^{\mathbbm kf}$ be the homotopy fibre of the composite $\mathbbm kf:X\xrightarrow{f} Y\xrightarrow{\mathbbm k}Y$, where $\mathbbm k$ is the $k$-th power map. Let also $F^g$ be the homotopy fibre of a map $g: X \to \prod_{i=1}^r Y$. The following result shows that for certain maps we can obtain equivalences among their homotopy fibres. 

\begin{theorem}\label{t:SEdec}
Let $X$ be a pointed topological space, $Y$ be a homotopy commutative $H$-group and $f:X\rightarrow Y$ be a map. Let $m=o(f)$, if $f$ has finite order, and $m=0$, otherwise. Suppose $g=\sum_{i=1}^r k_if_i\in \langle f_1,\dots,f_r\rangle \leq [X,Y^r]$ for some $r \geq 2$ and let  $k=\gcd(m,k_1,\dots,k_n)$. Then there is a homotopy decomposition 
\begin{equation}
F^g\simeq F^{\mathbbm kf}\times\prod_{i=1}^{r-1}\Omega Y.
\end{equation}
In particular, if the order of $f$ is finite and the maps $g,g'\in \langle f_1,\dots,f_n\rangle$ have the same order then $F^{g} \simeq F^{g'}$.
\end{theorem}

Sphere bundles over spheres and their connected sums appear frequently in classification problems of manifolds. In the context of the homotopy theory of gauge groups, a natural direction to investigate is, therefore, the one related to gauge groups over manifolds that arise as connected sums of sphere bundles. Our work is mainly concerned with odd dimensional manifolds given as connected sums of $S^q$-bundles over $S^n$ of dimension at least $7$ that admit cross sections. Let $j_*:\pi_k(SO(q))\to \pi_k(SO(q+1))$ be the map induced by the the inclusion $SO(q)\hookrightarrow SO(q+1)$.  Given a group $G$, let $|G|_{min}$ be the cardinality of a minimal generating set of $G$.

\begin{theorem}\label{t:UGG}
 Let $k,k',m\in\mathbb N$ such that $k+k'\leq m$, and $G$ be a Lie group. Let $\mathcal F=\{M_i\}_{i=1}^r$ be a family of manifolds that arise as total spaces of $S^{q}$-bundles over $S^{n}$ with characteristic elements $\chi(M_i)=j_*\xi^i$, $\xi^i\in\pi_{n-1}(SO(q))$ and let $M=\sharp_{i=1}^rM_i$, $r\geq2$, be a connected sum. Let $t=|\pi_{n+q}(S^q)|_{min}$.  Suppose that one of the following conditions holds: 
\begin{enumerate}
\item $G=SU(m)$, $n=2k$, $q=2k'-1$, $2 \leq k' \leq k$;
\item $G=Sp(m)$, $n=4k$, $q=4k'-1$, $1 \leq k' \leq k$.
\end{enumerate}
Then, given a principal $G$-bundle over $M$ classified by $K=(k_1,\dots,k_r)\in \prod_{i=1}^r\pi_{n-1}(G)$, there exists an integer $\bar t$, $0\leq\bar t\leq min\{r,t\}$, and a map $S^{n+q-1}\xrightarrow{\alpha}\bigvee_{i=1}^{\bar t}S^q$ with homotopy cofibre $Y_\mathcal F$ such that the gauge group $\mathcal G^K(M)$ decomposes, up to homotopy, as follows 
 \begin{equation}\label{decomp}
 \mathcal G^{K}(M)\simeq \mathcal G^{\ell(K)}(S^n)\times \prod_{i=1}^{r-1}\Omega^{n}G\times \prod_{j=1}^{r-\bar t}\Omega^{q}G\times \mathrm{Map}_*(Y_{\mathcal F},G),
 \end{equation}
where $\ell(K)=\gcd(o(\delta_1),k_1,\dots,k_r)$ and $\delta_1$ is the connecting map in the  fibration $$\Omega^n_1BG\longrightarrow{\rm{Map}}^1(S^{n},BG)\xrightarrow{ev} BG.$$
Moreover, in the case \label{i:n4q3} $n = 4$, $q = 3$, the homotopy splitting \eqref{decomp} holds for any simply connected simple compact Lie group $G$ whenever $\gcd(|\pi_6(G)|,\xi^1,\dots,\xi^r)=1$, $\xi^i\in\pi_3(SO(3))\cong\mathbb Z$. In this case $\bar t=1$ and $S^6\xrightarrow{\alpha}S^3$ is a unit in $\pi_6(S^3)\cong\mathbb Z_{12}$.
\end{theorem}
Theorem \ref{t:UGG} gives a general decomposition of the gauge groups provided that the rank of the Lie group and the dimensions of the spheres satisfy certain conditions. This decomposition shows that the homotopy types of the gauge groups $\mathcal G^K(M)$ depend only on the homotopy types of the gauge groups over $S^n$. Finally, the last part of Theorem \ref{t:UGG} can be regarded as an extension of \cite[Theorem 1.1]{IMS} to connected sums of $S^3$-bundles over $S^4$.

The paper is organised as follows. In section \ref{s:self-eq} we study the ring of power maps defined on $H$-spaces and the group structure of matrix maps. We prove Theorem \ref{t:GLn} and Theorem \ref{t:SEdec}. In section \ref{s:conn-sums} we collect some facts on sphere bundles over spheres and give some results on their suspensions. In section \ref{s:prin-G} we give some results on the classification of isomorphism classes of principal $G$-bundles over connected sums of $S^q$ bundles over $S^n$. In section \ref{s:gauge-groups} we study the homotopy theory of pointed and unpointed gauge groups. We prove Theorem \ref{t:UGG} and give a classification result for $SU(2)$-gauge groups over certain connected sums of $S^3$-bundles over $S^4$. 

\section{Self-equivalences of $H$-spaces} \label{s:self-eq}
In this work we assume that all spaces have the homotopy type of $CW$-complexes with non-degenerate basepoints. A pointed topological space $Y$ is an \emph{$H$-space} if there exists a map $\mu:Y\times Y\to Y$, called a homotopy multiplication, such that if $*:Y\to Y$ is the constant map to the basepoint, then the diagrams 
$$\xymatrix{
Y\ar[r]^-{(\mathbbm{1},*)}\ar[rd]_-{\mathbbm{1}_Y}&Y\times Y\ar[d]^-{\mu}\\
&Y
} \qquad  \qquad \xymatrix{
Y\ar[r]^-{(*,\mathbbm{1})}\ar[rd]_-{\mathbbm{1}_Y}&Y\times Y\ar[d]^-{\mu}\\
&Y
}$$
commute up to homotopy. We write $\mu(x,y):=xy$. An $H$-space is \emph{homotopy associative} with respect to $\mu$ if the maps $\mu\circ(\mu\times\mathbbm 1),\mu\circ(\mathbbm 1\times\mu):Y\times Y\times Y\to Y$ given by $(x,w,z)\mapsto (xw)z$ and $(x,w,z)\mapsto x(wz)$ are homotopic. Let $\Delta:X\to X\times X$ be the diagonal map. A \emph{homotopy inverse} of an $H$-space is a map  $\iota:Y\to Y$ such that the maps $\mu\circ(\iota\times\mathbbm 1)\circ\Delta:X\to X$ and $\mu \circ (\mathbbm 1\times\iota) \circ \Delta:X\to X$ are nullhomotopic. We write $\iota(x)=x^{-1}$. A homotopy associative $H$-space with a homotopy inverse is called an \emph{$H$-group}.  Let $T:Y\times Y\to Y\times Y$ be the map defined by $(x,y)\mapsto (y,x)$. An $H$-space $Y$ is \emph{homotopy commutative} if the maps $\mu\circ T,\mu:Y\times Y\to Y$ are homotopic.

Let $k\in \mathbb Z$ with $k \geq 1$. Given an $H$-group $Y$ with homotopy multiplication $\mu$, the $r$-fold product $Y^r$  has an $H$-group structure with a canonical coordinate-wise multiplication. For an $H$-group $Y$, we define the $k$-fold diagonal map $\Delta^k$ and the $k$-fold multiplication $\mu^k$ inductively:
\begin{equation}
\xymatrix{
\Delta^k: Y \ar[r]^-\Delta & Y \times Y\ar[rr]^-{\Delta^{k-1}\times \mathbbm{1}}&&Y^k,
}
\end{equation}
\begin{equation}
\xymatrix{
\mu^k: Y^k \ar[rr]^-{\mu^{k-1}\times\mathbbm 1} && Y\times Y  \ar[r]^-{\mu}&Y
}
\end{equation}
for $k \geq 2$, and $\Delta^1 = \mu^1 = \mathbbm{1}: Y \to Y$. Observe that $\mu=\mu^2$ and $\Delta=\Delta^2$. 
The $k$-th power map $\mathbbm k:Y\to Y$ is defined as the composite $\mathbbm k = \mu^k \circ \Delta^k$. For the negative number $-k\in\mathbb Z$, the $-k$-th power map $-\mathbbm k:Y\to Y$ is defined by ${-\mathbbm k} = \iota \circ \mu^{k} \circ \Delta^{k}$; here $\iota: Y \to Y$ is the homotopy inverse. For $k=0$, we define the zero power map $0: Y \to Y$ to be trivial map to the basepoint. Let $\mathcal P(Y)\subseteq[Y,Y]$ be the set of homotopy classes of power maps of $Y$.
Given a discrete group $G$, let $o(g)$ be the order of $g\in G$. 

\begin{lemma}\label{l:ring}
Let $Y$ be an $H$-group.  The set $\mathcal P(Y)$ has a ring structure isomorphic to  $\mathbb Z$, if $o(\mathbbm1_Y)=\infty$, or $\mathbb Z_m$, if $o(\mathbbm 1_Y)=m\in\mathbb N$.\qed
\end{lemma}
\begin{proof}
Clearly the map $\theta:\mathbb Z\to \mathcal P(Y)$ sending $k$ to the class of the $k$-th power map is a surjection. Given two elements $\mathbbm{a,b}\in\mathcal P(Y)$, we define their sum $\mathbbm a+\mathbbm b\in\mathcal P(Y)$ as the composite
\begin{equation}
\xymatrix{
\mathbbm a+\mathbbm b: Y \ar[r]^-\Delta & Y \times Y\ar[r]^-{\mathbbm a\times \mathbbm b} &Y\times Y\ar[r]^-{\mu}&Y.
}
\end{equation}
Composition of power maps, $\mathbbm a\circ \mathbbm b$, defines a second binary operation in $\mathcal P(Y)$. 

  We claim that if $Y$ is an $H$-group then $\theta$ is a ring homomorphism. First suppose that $a,b> 0$. We have 
  \begin{eqnarray*}
  \theta(a+b)(x)&=&\mu^{a+b}\circ\Delta^{a+b}(x)\\
  &=&\mu^{a+b}(x,\dots,x)\\
  &=&\mu(\mu(\cdots\mu(\mu(x,x),x)\cdots ,x),x)\\
  &=&(\cdots((\overbrace{xx)x)\cdots x)x}^{a+b}
    \end{eqnarray*}
and 
 \begin{eqnarray*}
(\theta(a)+\theta(b))(x)&=&\mu\circ(\theta(a)\times\theta(b))\circ\Delta(x)\\
&=&\mu(\mu^{a}\circ\Delta^{a}(x),\mu^{b}\circ\Delta^{b}(x))\\
&=&\mu(\mu^{a}(x,\dots, x),\mu^b(x,\dots,x))\\
&=&(\cdots((\overbrace{xx)x)\cdots x)x}^{a}(\cdots((\overbrace{xx)x)\cdots x)x}^{b}.
   \end{eqnarray*}
Using the homotopy associaivity of $\mu$ we can reorder brackets to show that the map $\theta(a+b)$, defined by $$x\mapsto (\cdots((\overbrace{xx)x)\cdots x)x}^{a+b}$$ is homotopic to the map $\mathbbm a+\mathbbm b$ 
defined by $$x\mapsto (\cdots((\overbrace{xx)x)\cdots x)x}^{a}(\cdots((\overbrace{xx)x)\cdots x)x}^{b}.$$
 In particular, it follows that $\mathbbm a+\mathbbm b = \theta(a+b) = \theta(b+a)=\mathbbm b+\mathbbm a$. 
  Observe that $\theta(0)=0$ represents the class of the constant map $*:Y\to Y$ and $\mathbbm a+0=\mathbbm a$. It is easy to verify that $\theta(a+b)=\theta(a)+\theta(b)$ for any $a,b\in\mathbb Z$ and $\mathbbm a+(-{\mathbbm a})=0$.

Now consider the composition of power maps
\begin{eqnarray*}
  \theta(ab)(x)&=&\mu^{ab}\circ\Delta^{ab}(x)\\
  &=&\mu^{ab}(x,\dots,x)\\
  &=&\mu(\mu(\cdots(\mu(\overbrace{x,x),x)\cdots),x),x}^{ab})\\
    &=&(\cdots((\overbrace{xx)x)\cdots x)x}^{ab}
     \end{eqnarray*}
and
\begin{eqnarray*}
  \mathbbm{a} \circ \mathbbm{b}(x)&=&\mu^a\circ \Delta^{a}\circ\mu^{b}\circ\Delta^{b}(x)\\
  &=&\mu^{a}\circ\Delta^{a}(\mu^b(x,\dots,x))\\
  &=&\mu^{a}(\mu^b(x,\dots,x),\dots,\mu^b(x,\dots,))\\
   &=&[\cdots[(\cdots((\underbrace{\overbrace{xx)x)\cdots x)x}^{b}(\cdots((\overbrace{xx)x)\cdots x)x}^{b}]\cdots](\cdots((\overbrace{xx)x)\cdots x)x}^{b}}_a.
  \end{eqnarray*}
 Using the homotopy associativity of $\mu$ we  reorder brackets to show that $\theta(ab)\simeq \mathbbm{a} \circ \mathbbm{b}$. In particular, it follows that $\mathbbm{a} \circ \mathbbm{b} \simeq \mathbbm{b} \circ \mathbbm{a}$. Thus $\theta$ is a ring homomorphism, as claimed. Finally, note that any ring on which $\mathbb{Z}$ surjects to via a ring homomorphism is isomorphic to either $\mathbb{Z}$ or $\mathbb{Z}_m$ for some $m \in \mathbb{N}$. This shows that $\mathcal P(Y)$ is has a ring structure isomorphic to $\mathbb Z$ or $\mathbb Z_m$. 
 \end{proof}
 
Let $(Y,\mu)$ be a homotopy associative $H$-space. Given a matrix  \begin{equation}
A=(a_{ij})=\begin{pmatrix} 
a & b \\
c & d
\end{pmatrix} \in \mat_2(\mathbb Z),
\end{equation} we define the \textit{matrix map} $\mathcal P_A:Y\times Y\to Y\times Y$ as the composite
\begin{equation}
\xymatrix{
Y\times Y\ar[r]^-{\Delta_{}}&Y\times Y\times Y\times Y\ar[rr]^-{{\mathbbm a}\times \mathbbm b \times{\mathbbm c}\times{\mathbbm d}}&&Y\times Y\times Y\times Y\ar[r]^-{\mu\times\mu}&Y\times Y,
}
\end{equation}
where $\mathbbm {a,b,c,d}\in\mathcal P(Y)$.  Observe 
 that the power maps associated to the coefficients of $A$ can be recovered from $\mathcal P_A$:  $$\mathbbm{a}_{ij}=p_{Y_i}\circ \mathcal P_A\circ i_{Y_j},$$ for $i,j\in\{1,2\},$ and $\mathcal P_A(x,y)=\mu(\mathbbm a(x),\mathbbm b(y))\times\mu(\mathbbm c(x),\mathbbm d(y))$.
 
More generally, let $r\geq 2$. Given a matrix $A\in \mat_r(\mathbb Z)$ we define the map $ \mathcal P_A$  as the composite
\begin{equation}
\xymatrix{
\prod_{j=1}^rY\ar[r]^-{\Delta^r}&\prod_{i=1}^r\prod_{j=1}^rY\ar[rr]^-{\prod_{i=1}^r\prod_{j=1}^r \mathbbm a_{ij}}&&\prod_{i=1}^r \prod_{j=1}^rY\ar[rr]^-{\prod_{i=1}^r\mu^r}&&\prod _{i=1}^rY
}
\end{equation}
Thus for any matrix $A\in\mat_r(\mathbb Z)$ there is a self-map $\mathcal P_A:Y^r\to Y^r$ associated to it. We want to construct self-equivalences out of matrix maps. In order to do so, we  restrict this construction to $GL_r(\mathbb Z)$, the subgroup of $\mat_r(\mathbb Z)$ whose elements are invertible matrices. In this case, we write $\mathcal E_A:=\mathcal P_A$ whenever $A\in GL_r(\mathbb Z)$. For any $r\geq2$, let $\mathcal E_{\mat}(Y^r)$ be the set of homotopy classes of matrix maps $\mathcal E_A$.

In order to give a prove Theorem \ref{t:GLn} we will require a lemma regarding the action of $GL_r(\mathbb Z_m)$ on $\mathbb Z_m^r$.
Let $X$ be a topological space and $f:X\to Y$ be a map. Let $f_i:X\to \prod _{i=1}^r Y$, for $1\leq i\leq r$, be the map such that $$p_j\circ f_i=\begin{cases} f & j=i\\ 0 &j\neq i\end{cases}$$ where $p_j$ is the projection onto the the $j$-th factor. Given $g\in G$, let $o(g)$ be the order of $g$. The \textit{greatest common divisor} of a set $S=\{a_1,\dots,a_r\}$ of integer numbers is denoted $\gcd(a_1,\dots,a_r)$. The greatest common divisor of $x_1,\dots,x_r\in\mathbb Z_m$ is defined by $${\gcd}_m(x_1,\dots,x_r):=\gcd(\tilde x_1,\dots,\tilde x_r,m),$$ where $\tilde x_i$ is a representative of the class of $x_i$. Here $\mathbb Z_0=\mathbb Z$ and $\gcd_0(a_1,\dots,a_r)=\gcd(a_1,\dots,a_r)$ for $a_1,\ldots,a_r \in \mathbb{Z}$.

\begin{lemma}\label{l:actionGLn}
Let $G=\mathbb Z^r_m$, $r\geq 2$ and $m,t\geq 0$. If $m$ divides $t$ then under the canonical action of $GL_r(\mathbb Z_t)$ on $G$, for any $g=(g_1,\dots,g_r)\in G$, the orbit of $g$ is the set $$\mathcal O_g=\{(h_1,\dots,h_r)\in G \mid {\gcd}_m(g_1,\dots,g_r)={\gcd}_m(h_1,\dots,h_r)\}.$$
\end{lemma}

\begin{proof}
Given $x\in \mathbb Z_m^r$ we can write $x=(x_1,\dots,x_r)$ with $x_i\in \mathbb Z_m$. If $m$ divides $t$ then for every $x\in \mathbb Z_m$ and $w\in \mathbb Z_t$ the product $w\cdot x \in \mathbb Z_m$ is well defined. Writing $x = (x_1,\ldots,x_r) \in \mathbb{Z}_m^r$ as a column vector, we define the action $\theta: GL_r(\mathbb Z_t)\times G\to G$ by the usual matrix multiplication. We need to show that given two elements $x,y\in G$, there exists an element $A\in GL_r(\mathbb Z_t)$ such that $Ax=y$, if and only if ${\gcd}_m(x_1,\ldots,x_r)={\gcd}_m(y_1,\ldots,y_r)$, where $(x_1,\ldots,x_r)=x$ and $(y_1,\ldots,y_r) = y$. 

The `only if' part is clear. Indeed, if $Ax=y$ then it is clear that ${\gcd}_m(x_1,\ldots,x_r)$ divides $y_i$ for all $i$, and since it also must divide $m$ we know that $\gcd_m(x_1,\ldots,x_r)$ divides ${\gcd}_m(y_1,\ldots,y_r)$. But since $A \in GL_r(\mathbb{Z}_t)$ we also get $A^{-1}y=x$ for a matrix $A^{-1} \in GL_r(\mathbb Z_t)$, and so similarly ${\gcd}_m(y_1,\ldots,y_r)$ divides $\gcd_m(x_1,\ldots,x_r)$. Thus ${\gcd}_m(x_1,\ldots,x_r) = \gcd_m(y_1,\ldots,y_r)$, as required.

Conversely, suppose ${\gcd}_m(x_1,\ldots,x_r) = {\gcd}_m(y_1,\ldots,y_r) = d \in \mathbb{N}$. We first note that it is enough to consider the case where $y$ is equal to $z = (d,0,\ldots,0) \in \mathbb{Z}_m^r$. Indeed, if $A,B \in GL_r(\mathbb{Z}_t)$ are matrices such that $Ax=z$ and $By=z$ then we have $Cx=y$ for $C = B^{-1}A \in GL_r(\mathbb{Z}_t)$. Thus, we may assume without loss of generality that $y = (d,0,\ldots,0)$. We will show that there exists a sequence of elementary row operations that takes the (column) vector $x \in \mathbb{Z}_m^r$ to $y \in \mathbb{Z}_m^r$; this implies the result. 

Suppose first that at most one $x_i$ is non-zero: that is, $x_1 = \cdots = x_{i-1} = x_{i+1} = \cdots = x_r = 0$ (in $\mathbb{Z}_m$) for some $i$. After swapping rows $2$ and $i$ in the column vector $x$ if necessary, we may assume without loss of generality that $i = 2$. Let $\tilde x_2 \in \mathbb{Z}$ be a representative of $x_2 \in \mathbb{Z}_m$. We then have $\gcd(\tilde x_2,m) = \gcd_m(x_1,\ldots,x_r) = d$, and so by B\'ezout's identity there exist $a,b \in \mathbb{Z}$ such that $a\tilde x_2+bm = d$. We may thus add $a \times \text{(the second row)}$ to the first row of $x$, transforming $x$ into $(d,x_2,0,\ldots,0) \in \mathbb{Z}_m^r$. As $\tilde x_2$ is a multiple of $d$, after adding a multiple of the first row to the second one we may transform the latter vector into $y = (d,0,\ldots,0)$, as required.

Finally, suppose that at least two $x_i$ are non-zero. For $1 \leq i \leq r$, let $\tilde x_i \in \mathbb{Z}$ be any representative of $x_i \in \mathbb{Z}_m$, and let $i \neq j$ be such that $0 < |\tilde x_i| \leq |\tilde x_j|$. The result then follows by induction on $N = \sum_{i=1}^r |\tilde x_i|$. Indeed, there exists $\varepsilon \in \{-1,1\}$ such that $|\tilde x_j + \varepsilon \tilde x_i| < |\tilde x_j|$. We may then add (if $\varepsilon = 1$) or subtract (if $\varepsilon = -1$) the $i$-th row from the $j$-th row of $x$; but then $|\tilde x_j + \varepsilon \tilde x_i| + \sum_{k \neq j} |\tilde x_k| < \sum_{k=1}^r |\tilde x_k| = N$, and so we are done by induction.
\end{proof}

We now give a proof of Theorem \ref{t:GLn}.
\begin{proof}[Proof of Theorem \ref{t:GLn}]
Let $Y$ be homotopy commutative $H$-group. By Lemma \ref{l:ring} the map $\theta: \mathbb Z \to \mathcal{P}(Y)$ given by $\theta(a)=\mathbbm a$ is a surjective ring homomorphism. A generating set of the group $GL_r(\mathbb Z_m)$ is given by the matrices $Q$ and $T$ 
\begin{equation}
Q=\begin{pmatrix} 
1 & 0& 0 &\cdots& 0 \\
1 & 1  & 0& \cdots& 0\\
0 & 0  & \ddots& \ddots& \vdots\\
\vdots& &\ddots&1&0\\
0 & 0& \cdots& 0 & 1
\end{pmatrix},\qquad
T=(-1)^{r-1}\begin{pmatrix} 
0 & 0& 0 &\cdots& 1 \\
1 & 0 & 0& \cdots& 0\\
0 & 1  & \ddots& \ddots& \vdots\\
\vdots& &\ddots&0&0\\
0 & 0& \cdots& 1 & 0
\end{pmatrix}.
\end{equation}
Let us look at the case $r=2$. We prove for the case $r=2$ that if $A,A'\in GL_r(\mathbb Z)$ with associated maps $\mathcal {E}_A, \mathcal E_{A'}\in \mathcal E_\mat(Y^r)$ then $\mathcal E_{A'A}\simeq\mathcal E_{A'}\circ\mathcal E_A.$ Since $Q$ and $T$ are generators of $GL_r(\mathbb Z)$ it suffices to show that  for any $A\in GL_2(\mathbb Z)$, $\mathcal E_{QA}\simeq\mathcal E_Q\circ \mathcal E_A$ and $\mathcal E_{TM}\simeq \mathcal E_T\circ\mathcal E_A$. 

Let $n=2$ and $A\in GL_2(\mathbb Z)$. Then $$QA=\begin{pmatrix} 
1 & 0 \\
1 & 1 
\end{pmatrix}\begin{pmatrix} 
a& b \\
c & d
\end{pmatrix}=\begin{pmatrix} 
a & b \\
a+c &b+d
\end{pmatrix}
$$
$$TA=\begin{pmatrix} 
0 & -1 \\
-1 & 0 
\end{pmatrix}\begin{pmatrix} 
a & b \\
c & d
\end{pmatrix}=\begin{pmatrix} 
-c & -d \\
-a &-b
\end{pmatrix}.
$$
First we show that $\mathcal E_{QA}\simeq\mathcal E_Q\circ \mathcal E_A$, that is, the maps defined by $(x,y)\mapsto (x^{a}y^{b},x^{a+c}y^{b+d})$ and $(x,y)\mapsto (x^{a}y^{b},(x^{a}y^{b})(x^cy^d))$ are homotopic. It suffices to show that the maps 
\begin{equation}\label{eq:matmap1}
(x,y)\mapsto x^{a+c}y^{b+d}
\end{equation}
\begin{equation}\label{eq:matmap2}
(x,y)\mapsto (x^{a}y^b)(x^cy^d)
\end{equation}
 are homotopic. Observe that map the map defined in \eqref{eq:matmap2} is the composition \begin{equation}
\xymatrix{
Y\times Y\ar[r]^-{\Delta_{}}&Y\times Y\times Y\times Y\ar[rr]^-{{\mathbbm a}\times \mathbbm b \times{\mathbbm c}\times{\mathbbm d}}&&Y\times Y\times Y\times Y\ar[r]^-{\mu\times\mu}&Y\times Y\ar[r]^-{\mu}&Y,
}
\end{equation} whereas the map \eqref{eq:matmap1} is the composition \begin{equation}
\xymatrix{
Y\times Y\ar[r]^-{\Delta_{}}&\prod_{i=1}^4 Y\ar[rr]^-{{\mathbbm a}\times \mathbbm b \times{\mathbbm c}\times{\mathbbm d}}&&\prod_{i=1}^4Y\ar[rr]^-{\mathbbm1\times T\times\mathbbm 1}&&\prod_{i=1}^4Y\ar[r]^-{\mu\times\mu}&Y\times Y\ar[r]^-{\mu}&Y.
}
\end{equation}
We have to show that the maps $\mu\circ(\mu\times\mu)$ and $\mu\circ(\mu\times\mu)\circ(\mathbbm1\times T\times \mathbbm 1)$ are homotopic. By assumption, $Y$ is a homotopy commutative $H$-group and therefore the following diagram 
\begin{equation}\label{d:Hmap}
\xymatrix{
Y \times Y \times Y \times Y \ar[rrrr]^{\mu \times \mu} \ar[drr]_{\id \times \id \times \mu} \ar[ddr]_{\id \times \mu \times \id} \ar[dddd]_{\id \times T \times \id} &&&& Y \times Y \ar[ddr]^\mu \\
&& Y \times Y \times Y \ar[urr]_{\mu \times \id} \ar[dr]^{\id \times \mu} \\
& Y \times Y \times Y \ar[rr]^{\id \times \mu} && Y \times Y \ar[rr]^\mu && Y \\
&& Y \times Y \times Y \ar[drr]^{\mu \times \id} \ar[ur]_{\id \times \mu} \\
Y \times Y \times Y \times Y \ar[rrrr]^{\mu \times \mu} \ar[urr]^{\id \times \id \times \mu} \ar[uur]^{\id \times \mu \times \id} &&&& Y \times Y \ar[uur]_\mu
}
\end{equation}
homotopy commutes. This shows that $\mu\circ(\mu\times\mu)\simeq\mu\circ(\mu\times\mu)\circ(\mathbbm1\times T\times \mathbbm 1)$. For the general case, it suffices to show that the maps $Y^r \to Y$ defined by $(x_1,\ldots,x_r)\mapsto (x_1^{a_1} \cdots x_r^{a_r})(x_1^{b_1}\cdots x_r^{b_r})$ and $(x_1,\ldots,x_r)\mapsto x_1^{a_1+b_1} \cdots x_r^{a_r+b_r}$ are homotopic; here $(a_1,\dots,a_r)$ and $(b_1,\dots,b_r)$ are the first two row vectors of the matrix $A\in GL_r(\mathbb Z)$. We can construct a diagram similar to \eqref{d:Hmap} and use homotopy commutativity and homotopy associativity of the $H$-group $Y$ to show that this diagram commutes.

To show that $\mathcal E_{TA}\simeq \mathcal E_T\circ\mathcal E_A$, note that the map $\mathcal E_{TA}$ is defined by $(x,y)\mapsto (x^{-c}y^{-d})(x^{-a}y^{-b})$, whereas $\mathcal E_T\circ\mathcal E_A$ is defined by $(x,y)\mapsto ((x^c)^{-1}(y^d)^{-1})((x^a)^{-1}(y^b)^{-1})$. The result follows from the fact that the map $\theta: \mathbb Z \to \mathcal{P}(Y)$, defined by sending $k$ to $k$-th power map, is a ring homomorphism. The general case follows similarly. This proves the first part of the theorem.

To show the second part of the theorem, note that by the first part, every map in $\mathcal{E}_\mat(Y^r)$ is a composite of maps $\mathcal{E}_Q$, $\mathcal{E}_T$ and their inverses. Since composites and inverses of $H$-maps are also $H$-maps, it is therefore enough to show that $\mathcal{E}_Q$ and $\mathcal{E}_T$ are $H$-maps. This follows easily for $\mathcal{E}_T$. For $\mathcal{E}_Q$, note that the maps $\mathcal{E}_Q \circ \overline{\mu}, \overline{\mu} \circ (\mathcal{E}_Q \times \mathcal{E}_Q): Y^r \times Y^r \to Y^r$, where $\overline{\mu}$ is the induced multiplication in $Y^r$, are given by
\[
((x_1,\ldots,x_r),(y_1,\ldots,y_r)) \mapsto (x_1y_1,(x_1y_1)(x_2y_2),x_3y_3,\ldots,x_ry_r)
\]
and
\[
((x_1,\ldots,x_r),(y_1,\ldots,y_r)) \mapsto (x_1y_1,(x_1x_2)(y_1y_2),x_3y_3,\ldots,x_ry_r),
\]
respectively. It is therefore enough to show that the map $(x_1,x_2,y_1,y_2) \mapsto (x_1y_1)(x_2y_2)$ is homotopic to the map $(x_1,x_2,y_1,y_2) \mapsto (x_1x_2)(y_1y_2)$. But this follows from \eqref{d:Hmap}.

For part three of the theorem suppose that $\pi_0(Y)=\mathbb Z_m$ and let $K,K'\in\pi_0(Y^r)=\mathbb Z_m^r$. Suppose that $\gcd_m(k_1,\dots,k_r)=\gcd_m(k_1'\dots,k_r')$, where $(k_1,\dots,k_r) = K$ and $(k_1',\dots,k_r') = K'$. Then, by Lemma \ref{l:actionGLn}, there exists a matrix $A \in GL_r(\mathbb{Z}_t)$, where $t$ is as in the first part of the theorem, such that $AK = K'$. In particular, $\mathcal{E}_A$ maps $(Y^r)_K$ to $(Y^r)_{K'}$. But by the first part of the theorem we have $\mathcal{E}_A \circ \mathcal{E}_{A^{-1}} \simeq \mathcal{E}_{I_n} = \mathbbm{1}_{Y^r} \simeq \mathcal{E}_{A^{-1}} \circ \mathcal{E}_A$, and so $\mathcal{E}_A$ is a homotopy equivalence. It follows that $\mathcal{E}_A: (Y^r)_K \rightarrow (Y^r)_{K'}$, is a homotopy equivalence as required. Conversely, if $\gcd_m(k_1,\dots,k_r) \neq \gcd_m(k_1'\dots,k_r')$ then, by Lemma \ref{l:actionGLn}, there are no invertible matrices $A \in GL_r(\mathbb{Z}_t)$ for which $\mathcal{E}_A$ maps $(Y^r)_K$ to $(Y^r)_{K'}$.
\end{proof}
In 
Let $f:X\to Y$ be a map. For $1\leq i\leq r$, let $f_i:X\to \prod _{i=1}^r Y$ be the map such that $$p_j\circ f_i=\begin{cases} f & j=i\\ 0 &j\neq i\end{cases}$$ where $p_j$ is the projection onto the the $j$-th factor. Given a group $G$ and elements $g_i\in G$, $1\leq i\leq r$, let $\langle g_1,\dots ,g_r\rangle$ be the group generated by those elements. Given a map $f:X\rightarrow Y$, we denote its homotopy class by the same letter $f$. 

\begin{proposition}\label{p:linearaction}
Let $X$ be a pointed connected topological space, $Y$ be a homotopy commutative $H$-group and $f:X\rightarrow Y$ be a map with order $m$. The group $\mathcal E_A(Y^r)$, $r\geq 2$, acts linearly on the set $ \langle f_1,\dots,f_r\rangle \leq [X,Y^r]$. In particular, if $g=\sum_{i=1}^r k_if_i$ and $g'=\sum_{i=1}^r k'_if_i$ are maps such that $\gcd(m,k_1,\dots,k_r)=\gcd(m,k'_1\dots,k_r')$, then there is a homotopy $\mathcal E_A\circ  g\simeq g'$, for some $\mathcal E_A\in\mathcal E_\mat(Y^r)$.
\end{proposition}

\begin{proof}
We prove the proposition in the case $r=2$; the general case follows similarly. Let $A = (a_{ij}) \in GL_2(\mathbb Z)$ and $g=(k_1f,k_2f)\in\langle f_1,f_2\rangle$.  We show that for any $\mathcal E_A\in \mathcal E_\mat(Y^2)$ 
\begin{equation}\label{e:haGLn}
\mathcal E_A\circ g\simeq (\mathbbm a_{11} k_1f+\mathbbm a_{12}k_2 f,\mathbbm a_{21} k_1f+\mathbbm a_{22}k_2f),
\end{equation}
where $\mathbbm a kf+\mathbbm a' k'f:X\to Y$ is defined as the composite
\begin{equation}
X\xrightarrow{\Delta}X\times X\xrightarrow{\mathbbm a kf\times \mathbbm a' k'f}Y\times Y\xrightarrow{\mu} Y
\end{equation}
That is, the group $\mathcal E_A$ acts on $\langle f_1,f_2\rangle$ in a similar manner as the canonical  action of $GL_2(\mathbb Z)$ on $\mathbb Z_m^2$, $m\geq 0$.

Consider the diagram
\begin{equation}
\xymatrix{
X\ar[rr]^-{\Delta}\ar[d]^-{\Delta}&&X\times X\ar[rrr]^-{k_1f\times k_2f}\ar[d]^-{\Delta}&&&Y\times Y\ar[d]^-{\Delta}\\
X\times X\ar[rr]^-{\Delta\times\Delta}\ar[ddrrrrr]_-{\alpha_1\times\alpha_2}&&\prod_{i=1}^4 X\ar[rrr]^-{\prod_{i=1}^2 (k_1f \times k_2f)}\ar[rrrd]_-{\beta}&&&\prod_{i=1}^4Y\ar[d]^-{\prod_{i=1}^2\prod_{j=1}^2a_{ij}}\\
&&&&&\prod_{i=1}^4Y\ar[d]^-{\mu\times\mu}\\
&&&&&Y\times Y
}
\end{equation}
where $\beta=\prod_{i=1}^2\prod_{j=1}^2 a_{ij} k_j f$ and $\alpha_i = a_{i1} k_1 f + a_{i2} k_2 f$. Observe that  composite in the direction along the upper row and right column in the diagram is the map $\mathcal E_A\circ \gamma$, whereas the composite  in the direction left vertical arrow  is the map $(\mathbbm a_{11} k_1f+\mathbbm a_{12}k_2 f,\mathbbm a_{21} k_1f+\mathbbm a_{22}k_2f)$. The upper squares commute pointwise. The middle triangle homotopy commutes by composition of power maps. The lower quadrilateral  homotopy commutes by looking at the projection at each factor. This shows  that the expression \eqref{e:haGLn} holds.

Now by Lemma \ref{l:actionGLn}, given $K,K'\in\mathbb Z_m^2$ such that $\gcd_m(k_1,k_2)=\gcd_m(k'_1,k_2')$ there is a matrix $A\in GL_2(\mathbb Z)$ such that $K=AK'$. Therefore if $g=\sum_{i=1}^2 k_if_i$ and $g'=\sum_{i=1}^2 k'_if_i$ are maps $ \langle f_1,f_2\rangle\cong\mathbb{Z}_m^r$ such that $\gcd(k_1,k_2,m)=\gcd(k'_1,k_2',m)$ we can construct a self-equivalence $\mathcal E_A$ such that $\mathcal E_A\circ  g\simeq g'$. \end{proof}

Let $F^{\mathbbm kf}$ denote the homotopy fibre of the composite $\mathbbm kf: X\xrightarrow{f} Y\xrightarrow{\mathbbm k}Y$, where $\mathbbm k:Y\to Y$ is the $k$-th power map. In \cite{Th2} it is proved the following result regarding the homotopy types of the fibres $F^{\mathbbm kf}$.

\begin{proposition}[Theriault, {\cite[Lemma 3.1]{Th2}}]\label{c:theriault}
Let $f:X\to Y$ be a map of finite order $m$. If $(m,k)=(m,k')$ then there is a homotopy equivalence $F^{\mathbbm kf}\simeq F^{\mathbbm k' f}$ when localised rationally or at any prime. 
 \end{proposition}
\begin{comment} 
 \begin{proof}
Localise all spaces rationally or at a prime $p$ such that $(m,p)=1$. Let $h,h':X\to Z$ be maps. Observe that if there exists a self-homotopy equivalence $\theta:Y\to Y$ such that $\theta\circ h\simeq h'$ then the bottom right square in \eqref{d:commfib} homotopy commutes. We can take homotopy fibres along each map of the bottom square to generate the whole diagram which homotopy commutes.  This implies that the map $\varphi$ is a homotopy equivalence.
 
 \begin{equation}\label{d:commfib}
\xymatrix{
F^h\ar@{=}[r]\ar[d]^{\varphi}&F^{h}\ar[r]\ar[d]&{*}\ar[d]\\
F^{h'}\ar[r]\ar[d]&X_{(p)}\ar[r]^-{h'}\ar[d]^-{h}&Y_{(p)}\ar@{=}[d]\\
{*}\ar[r]&Y_{(p)}\ar[r]^{\mathcal E_A}&Y_{(p)}
}
\end{equation}

 Since $f$ has finite order then the map $f$ is nulhomotopic. It follows that if $(m,k)=(m,k')$ then  $F_{\mathbbm k\circ f}\simeq F_{\mathbbm k'\circ f}$. Now localise at a prime $p$ such that $(m,p)=p$. Suppose $(m,k)=(m,k')$, then the maps $h=\mathbbm k\circ f$ and $h'=\mathbbm k'\circ f$ have the same order in the subgroup generated by $f\in[X,Y]$. Therefore by Lemma \ref{l:actionGLn}, there exists an element $\theta\in \mathcal E_M(Y)$ such that $\theta\circ h\simeq h'$. This implies that diagram \eqref{d:commfib} homotopy commutes and therefore $F_{\mathbbm k\circ f}\simeq F_{\mathbbm k'\circ f}$. 
\end{proof}
\end{comment}
Theorem \ref{t:SEdec} is an integral result motivated by Proposition \ref{c:theriault} in which we use  our construction of $\mathcal E_{\mat}(Y^r)$ for the cases when the $H$-space $Y$ is an $n$-fold product.  

\begin{proof}[Proof of Theorem \ref{t:SEdec}]
 Let $\tilde g=\mathbbm kf_1$. The group $\mathcal E_\mat(Y^r)$ acts on  $\langle f_1,\dots,f_r\rangle$ by composition.
 By Proposition \ref{p:linearaction} this action is linear and therefore there is an element $\mathcal E_A\in\mathcal E_\mat(Y^r)$ such that $g' \simeq \mathcal E_A\circ g$ making the bottom right square in diagram \eqref{d:decomp} commute up to homotopy. 
 
 \begin{equation}\label{d:decomp}
\xymatrix{
F^g\ar@{=}[r]\ar^-{\varphi}[d]&F^{g}\ar[r]\ar[d]&{*}\ar[d]\\
F^{g'}\ar[r]\ar[d]&X\ar[r]^-{g'}\ar[d]^-{g}&Y^r\ar@{=}[d]\\
{*}\ar[r]&Y^r\ar[r]^-{\mathcal E_A}&Y^r
}
\end{equation}
Taking homotopy fibres along each map in the  bottom square generates the whole diagram \eqref{d:decomp}, where each row and column is a homotopy fibration sequence.  Homotopy commutativity of the diagram implies that $\varphi$ is a homotopy equivalence. Identifying $X$ with $X\times *\times\cdots\times*$ is easy to see that there is a homotopy equivalence 
\begin{equation}
F^g\simeq F^{\mathbbm kf}\times\prod_{i=1}^{r-1}\Omega Y,
\end{equation}
as required.
\end{proof}

We recall some general properties of $H$-spaces.  If $Y$ is a homotopy associative $H$-space with a homotopy inverse then there are homotopy equivalences between all path-components of $Y$. In particular, if $Y_0$ is the path-component containing the basepoint of $Y$, there are homotopy equivalences $\Theta_{\alpha}:Y_{\alpha}\to Y_0$ and $\Psi_{\alpha}:Y_0\to Y_\alpha$ such that $\Psi_\alpha\circ\Theta_\alpha\simeq \mathbbm 1_{Y_\alpha}$ and  $\Theta_\alpha\circ\Psi_\alpha\simeq \mathbbm 1_{Y_0}$, for all path-components $Y_\alpha$. To see this, for each $Y_\alpha$, define  maps 
\begin{equation}\label{eq:theta}
\Theta_\alpha= m\circ(\tilde\alpha\times\mathbbm 1)\circ\Delta: Y \to Y
\end{equation}
\begin{equation}\label{eq:psii}
 \Psi_\alpha = m\circ (i(\tilde \alpha)\times\mathbbm 1)\circ\Delta: Y \to Y,
\end{equation}  
 where $\tilde\alpha$ is a fixed element of $Y_\alpha$. Observe that the maps $\Theta_\alpha$ and $\Psi_\alpha$ satisfy $\Theta_\alpha(Y_0) \subseteq Y_\alpha$ and $\Psi_\alpha(Y_\alpha) \subseteq Y_0$. We aim to show that $\Theta_\alpha$ and $\Psi_\alpha$ are homotopy inverses, thus inducing a homotopy equivalence $Y_0 \simeq Y_\alpha$. We show that $\Psi_\alpha\circ \Theta_\alpha \simeq \mathbbm{1}_Y$; the argument for $\Theta_\alpha \circ \Psi_\alpha$ is similar. Consider the homotopy commutative diagram
\begin{equation}
\xymatrix@=0.5in{
Y \ar[r]^-{(\tilde\alpha, \mathbbm{1}_Y)} \ar[d]^-{j_2} & Y \times Y \ar[rr]^-{(i, \mathbbm{1}_Y) \times \mathbbm{1}_Y} \ar[d]^-{* \times \mathbbm{1}_Y} && Y \times Y \times Y \ar[r]^-{\mathbbm{1}_Y \times m} \ar[d]^-{m \times \mathbbm{1}_Y} & Y \times Y \ar[d]^-{m} \\
Y \times Y \ar@{=}[r] & Y \times Y \ar@{=}[rr] && Y \times Y \ar[r]^-{m} & Y,
}
\end{equation}
 where $j_2: Y \to Y \times Y$ is the inclusion into the second factor, given by $j_2(\alpha) = (*,\alpha)$. Here the square on the right homotopy commutes because of homotopy associativity of the map $m$, and the middle one homotopy commutes because of properties of the homotopy inverse $i$. The composite of maps on the top and the right of the diagram is just $\Psi_\alpha \circ \Theta_\alpha = m(i(\tilde\alpha),m(\tilde\alpha,-))$, and the map $m \circ j_2$ along the left and the bottom of the diagram is homotopic to $\mathbbm{1}_Y$. Thus $\Psi_\alpha\circ \Theta_\alpha \simeq \mathbbm{1}_Y$. 

Given a homotopy commutative $H$-group $Y=\bigsqcup_{\alpha}Y_\alpha$, the set of path-components $\pi_0(Y)$ is an abelian group. In this case, homotopy classes of self-equivalences $\Psi_K:Y^r\to Y^r$ and $\Theta_K:Y^r\to Y^r$ are compatible with the elements $\mathcal E_A\in\mathcal E_M(Y^r)$ in the sense of Lemma \ref{l:theta-E_A}. 

\begin{lemma}\label{l:theta-E_A} Let $Y=\bigsqcup_{\alpha\in\pi_0(Y)}Y_\alpha$ be and $H$-group and let $\prod_{i=1}^rY_{\alpha_i}$ and $\prod_{i=1}^rY_{\alpha'_i}$ be path-components of $Y^r$ indexed by $K=(\alpha_1,\dots,\alpha_r),K'=(\alpha'_1,\dots,\alpha'_r)\in \pi_0(Y)^r$. Suppose that $K'=AK$, for some $A\in GL_r(\mathbb Z)$. Then the diagram
\begin{equation}
\xymatrix@=0.4in{
\prod_{i=1}^rY_{\alpha_i}\ar[r]^-{\Theta_K}\ar[d]^{\mathcal E_A}&\prod_{i=1}^rY_{0}\ar[d]^-{\mathcal E_A}\\
\prod_{i=1}^rY_{\alpha'_i}\ar[r]^-{\Theta_{K'}}&\prod_{i=1}^rY_{0}\
}
\end{equation}
commutes up to homotopy.
\end{lemma}
\begin{proof}
For the $H$-group $Y^r$, let $m:Y^r\times Y^r\to Y^r$ and $i:Y^r\to Y^r$ be the homotopy multiplication and a homotopy inverse, respectively. Consider the following diagram.
\begin{equation}
\xymatrix@=0.5in{
Y^r\ar[r]^-{\Delta}\ar[d]^{\mathcal E_A}&Y^r\times Y^r\ar[r]^{i(\tilde K)\times\mathbbm 1}\ar[d]^{\mathcal E_A\times\mathcal E_A}&Y^r\times Y^r
\ar[r]^-{m}\ar[d]^-{\mathcal E_A\times\mathcal E_A}&Y^r\ar[d]^-{\mathcal E_A}\\
Y^r\ar[r]^-{\Delta}&Y^r\times Y^r\ar[r]^-{i(\tilde K')\times\mathbbm 1}&Y^r\times Y^r\ar[r]^-{m}&Y^r
}
\end{equation}
The left square homotopy commutes by definition of  the diagonal map $\Delta$. To check homotopy commutativity of the middle square it suffices to check homotopy commutativity of the composites on each factor of the product. The projection onto the second factor homotopy commutes since $\mathcal E_A\circ\mathbbm 1\simeq\mathbbm 1\circ\mathcal E_A$.  Now we check homotopy commutativity after projecting onto the first factor. Since $\tilde K$ is a constant, $i(\tilde K)\in (Y^r)_K$ is also a constant. By assumption $K'=AK$, therefore, by the third part of Theorem \ref{t:GLn}  the composite $\mathcal E_A\circ i(\tilde K)$ has image in $\prod_{i=1}^rY_{-\alpha_i'}$. Similarly, since $\tilde K'$ is fixed, the lower composite $i(\tilde K')\circ\mathcal E_A$ is constant and, in particular with image in $\prod_{i=1}^rY_{-\alpha'_i}$. Thus the upper and lower composites in the first factor of the middle square are constant maps with images in the same path-component of $Y^r$. Therefore, the middle  square homotopy commutes. The right square homotopy commutes because by Theorem \ref{t:GLn} the map $\mathcal E_A$ is an $H$-map.
\end{proof}

Let $G$ be a connected simple compact Lie group. Isomorphism classes of principal $G$-bundles over a wedge sum $\bigvee_{i=1}^rS^n_i$ are classified by the homotopy set $[\bigvee_{i=1}^rS^n_i,BG]=\bigoplus_{i=1}^r\pi_{n-1}(G)$. If $\pi_{n-1}(G)=\mathbbm Z$ then we have
\begin{equation}\label{e:prinGwSn} 
 Prin_G(\bigvee_{i=1}^r S^n_i)=[\bigvee_{i=1}^rS^n_i,BG]=\mathbb Z^r.
 \end{equation}
Observe that $\pi_0(\map(\bigvee_{i=1}^rS^n_i,BG))=[\bigvee_{i=1}^rS^n,BG]$. The connected components of  the mapping space $\map(\bigvee^r_{i=1}S^n,BG)$ are classified by $r$-tuples of integers. Let $\map^K(\bigvee^r_{i=1}S^n,BG)$ be the connected component classified by $K=(k_1,\dots,k_r)\in\mathbb Z^r$. The restriction of the evaluation map $ev:\map(\bigvee_{i=1}^rS^n_i,BG)\to BG$, $ev(f)=f(*)$, to the $K$-th component  induces the following homotopy fibration
\begin{equation}\label{fib:evwedge}
\xymatrix{
\mathcal{G}_{K}\left(\bigvee_{i=1}^r S^n_i\right) \ar[r] & G \ar[r]^-{\tilde\partial_K} &\mapp^K(\bigvee^r_{i=1}S^n,BG)\ar[r]& \mathrm{Map}_K(\bigvee_{i=1}^r S^n_i, BG)\ar[r]^-{ev_K}&BG
}
\end{equation}
where we have identified $\mathcal{G}_{K}\left(\bigvee_{i=1}^r S^n_i\right)$ with  $\Omega\mathrm{Map}_K(\bigvee_{i=1}^r S^n_i, BG)$ and $\bar\partial^K$ is the connecting map. 
According to a result of Lang \cite{Lng}, the adjoint of the connecting map in the evaluation  fibration \eqref{fib:evwedge} is a Whitehead product. In the following lemma we state this result in terms of Samelson products. Let $\delta_{ij}=\begin{cases}
1&i=j\\
0&i\neq j 
\end{cases}$, and $\Theta_K$ be the map defined in \eqref{eq:theta}.
 \begin{lemma}\label{l:lang}~
Let $K=(k_1,\dots,k_r)\in\mathbb Z^r$. 
The adjoint $G\wedge\left(\bigvee_{i=1}^rS^{n-1}_i\right)\to G$ of the composite 
 $$\partial_K:G\xrightarrow{\tilde\partial_{K}}{\rm{Map}}^K_*(\bigvee_{i=1}^r S^{n-1}_i,G) \xrightarrow{\Theta_K}{\rm{Map}}^0_*(\bigvee_{i=1}^r S^{n-1}_i,G)\xrightarrow{\cong} \prod_{i=1}^r\Omega_0^{n-1}G$$
 is homotopic to the Samelson product $(\langle\mathbbm{1}_G,\mathbbm k_1\epsilon_1\rangle,\dots\langle\mathbbm{1}_G,\mathbbm k_r\epsilon_r\rangle)$ where $\delta_{ik}\epsilon=\epsilon_i\circ j_k$, $\epsilon$  is a generator of $\pi_{n-1}(G),$ and $j_k$ is the inclusion of $S^{n-1}_k$ into the wedge.
\qed
\end{lemma}

Linearity on Samelson products implies  $\langle\mathbbm{1}_G,\Sigma_{i=1}^r\mathbbm k_i\epsilon_i\rangle\simeq (\mathbbm k_1\langle\mathbbm 1_G,\epsilon_1\rangle,\dots,\mathbbm k_r\langle\mathbbm 1_G,\epsilon_r\rangle) $, and therefore we have the following corollary. Let $\partial_K$ be the composite defined in Lemma \ref{l:lang}. 
\begin{cor}\label{cor:lang}
There is a homotopy $\partial_K\simeq(\mathbbm k_1\partial_{1},\dots,\mathbbm k_r\partial_{1}).$\qed
\end{cor}

The following result states that under mild conditions on $G$ and $X$, there are only finitely many homotopy types associated to the set $Prin_G(X)$.
\begin{theorem}[Crabb-Sutherland {\cite{CS}}]\label{r:cs}
Let $X$ be a connected finite complex and let G be a compact connected Lie group. As $P$ ranges over all principal $G$-bundles over $X$, the number of homotopy types of $\mathcal G^P(X)$ is finite. 
\end{theorem}

We can use these results to obtain homotopy decompositions of gauge groups $$\mathcal G^K\left(\bigvee_{i=1}^r S^{n}_i\right)\simeq \Omega\map^K\left(\bigvee_{i=1}^r S^{n}_i,BG\right).$$ associated to the set of principal $G$ bundles given in \eqref{e:prinGwSn}.

If $G$ is a connected Lie group, then $\prod_{i=1}^{r}\Omega^n G$ is a homotopy commutative $H$-group. As the group $\bigoplus_{i=1}^r\pi_n(G)\cong\mathbb Z^r$ is torsion-free, the identity map $\mathbbm 1: \prod_{i=1}^{r}\Omega^n G\to \prod_{i=1}^{r}\Omega^n G$ has infinite order. By Theorem \ref{t:GLn} the subgroup of self-equivalences $\mathcal  E_\mat((\Omega^n G)^r)$ is isomorphic to $GL_r(\mathbb Z)$.

\begin{proposition}\label{p:wedgegg}
Let $K=(k_1,\dots,k_r)\in\mathbb Z^r$. There are homotopy equivalences
\begin{equation}
\mathcal G^{K}\left(\bigvee^r_{i=1}S^n\right)\simeq \mathcal G^{\ell(K)}\left(S^n\right)\times\prod_{i=1}^{r-1}\Omega^n(G),
\end{equation}
where $\ell(K)=gcd(o(\partial_1),k_1,\dots,k_r)$ and $\partial_1$ is the connecting map of the fibration sequence
$$\Omega^n_1BG\longrightarrow{\rm{Map}}^1(S^{n},BG)\xrightarrow{ev} BG.$$
 \end{proposition}
 \begin{proof}
By Corollary \ref{cor:lang} we have $\partial_K\simeq (\mathbbm k_1\partial_1,\dots\mathbbm  k_r\partial_1)$, where $K=(k_1,\dots,k_r)\in\mathbbm Z^r$. Since $\bigvee_{i=1}^r S^n$ is a finite connected $CW$-complex and $G$ is connected, by  Theorem \ref{r:cs}, $o(\partial _1)=m$, for some $m\in\mathbb N$. Let $\partial_{\mathbbm{\ell}(K)}:= (\mathbbm \mathbb \ell,0,\dots,0)$, where  $\ell=\gcd(m,k_1,\dots,k_r)$. Observe that $o(\partial_K)=o(\partial_{\ell(K)})$. Consider the diagram 
\begin{equation}\label{d:connmap_wedge}
\xymatrix{
G\ar[rr]^-{\bar\partial_K}\ar@{=}[d]&&\prod_{i=1}^{r}\Omega^n_{k_i}(G)\ar[rr]^-{\Theta_K}\ar[d]^-{\mathcal E_A}&&\prod_{i=1}^{r}\Omega^n_0(G)\ar[d]^-{\mathcal E_A}\\
G\ar[rr]^-{{\bar\partial_{\ell (K)}}}&&\Omega^n_{\ell}(G) \times \prod_{i=1}^{r-1}\Omega^n_{0}(G)\ar[rr]^-{\Theta_{\ell({K})}}&&\prod_{i=1}^{r}\Omega^n_0(G)
}
\end{equation}
The upper and lower rows in the diagram are the maps $\partial_K$ and $\partial_{\ell( K)}$, respectively. By Theorem \ref{t:SEdec} there is a self-equivalence $\mathcal E_A\in\mathcal E_\mat(\prod_{i=1}^{r}\Omega^n(G))$ making the big rectangle in  \eqref{d:connmap_wedge} homotopy commute. By Lemma \ref{l:theta-E_A} the right square homotopy commutes.  Since the the outer rectangle and the right square homotopy comute and $\Theta_K$ and $\Theta_{\ell(K)}$ are homotopy equivalences, the left square homotopy commutes. Since the conditions of Theorem \ref{t:SEdec} are satisfied, the result follows. 
\end{proof}
  
\section{Connected sums of sphere bundles over spheres} \label{s:conn-sums}

In this section we study the homotopy theory of connected sums of sphere bundles over spheres and their suspension. In order to do so, we will state two useful lemmas. These results will let us find homotopy splittings of suspensions of connected sum.

Let $R$ be a principal ideal domain and $H$ be a finitely generated $R$-module. Then $$H\cong R^{t'}\{\alpha_1,\dots,\alpha_{t'}\}\oplus R/{s_1}R\{\alpha_{t'+1}\}\oplus\cdots\oplus R/s_{t-t'} R\{\alpha_t\},$$
for some $0\leq t'\leq t$, $2\leq s_i$ and $s_j|s_{j+1}$ for all $j=1,\dots,t-t'-1$. Let $\Pi= H^{r}$ and $x\in\Pi$. If $x=(a_{11}\alpha_1+\cdots+a_{1t}\alpha_t,\dots, a_{r1}\alpha_1+\cdots+a_{rt}\alpha_t)$, then we can represent $x$ by an $r\times t$ matrix $A=(a_{ij})$ with entries in $R_j=R$ within the columns $j=1,\dots t'$ and entries in $R_j=R/s_j R$ within the column $j=t+i$, $i=1,\dots, t-t'$. Observe that this representation of elements in $\Pi$ makes $\Pi$ into a $GL_r(R)$-module. 

A matrix $A\in \mat_{r\times t}(R)$ is in a \textit{row echelon form} if it has the following properties:
\begin{enumerate}
\item If a row is non-zero, then the first non-zero entry (from the left) of the row is strictly to the right with respect to the first non-zero entries of the rows above it. 
\item zero rows are at the bottom of the matrix. 
\end{enumerate}

\begin{lemma}\label{l:matrixech}
For any $x\in \Pi$ there is a matrix $D\in GL_r(R)$ such that $DA$ is in a row echelon form, where $A$ is an $r\times t$ matrix representing $x$. 
\end{lemma}
\begin{proof}
The result follows by an iterative application of Lemma \ref{l:actionGLn}.
\end{proof}

\begin{lemma}\label{l:selfeqwedge}
There are self equivalences $$\theta_{D_1,D_2}:\bigvee_{i=1}^{r_1} S^n\vee\bigvee_{i=1}^{r_2} S^q\to \bigvee_{i=1}^{r_1} S^n\vee\bigvee_{i=1}^{r_2} S^q$$ which induce isomorphisms in homology. Moreover, the maps $(\theta_{D_1,D_2)}$ form a group isomorphic to a subgroup of $GL_{r_1+r_1}(\mathbb Z)$ with elements of the form
$$A=\begin{pmatrix} 
D_1 & 0 \\
0 & D_2 
\end{pmatrix}$$
where $D_1\in GL_{r_1}(\mathbb Z)$ and $D_2\in GL_{r_2}(\mathbb Z)$
\end{lemma} 
\begin{proof}
Let $n,q\geq 2$. Then $Y=\bigvee_{i=1}^{r_1} S^n\vee\bigvee_{i=1}^{r_2} S^q$ is a homotopy associative homotopy commutative co-$H$-space with comultiplication $\sigma:Y\to Y\vee Y$ given by pinching the equator in $\Sigma (\bigvee_{i=1}^{r_1} S^{n-1}\vee\bigvee_{i=1}^{r_2} S^{q-1})\cong Y$. Given a matrix $D=(d_{ij})\in GL_2(\mathbb Z)$, we define the map $\theta_D:\bigvee_{i=1}^2S^n_i\to\bigvee_{i=1}^2S^n_i$ as the composite 
\begin{equation}\label{map:tdeg}
\bigvee_{i=1}^2S^n_i\xrightarrow \sigma \bigvee_{i=1}^4S^n_{i}\xrightarrow {\bigvee_{i=1}^4 d_{ij}}\bigvee_{j=1}^4S^n_{j}\xrightarrow{\nabla}\bigvee_{j=1}^2S^n_{j}
\end{equation}
where the entries $d_{ij}$ in $D$ define the degree maps $d_{ij}:S^n_{i}\to S_j^n$.   Let $X=\bigvee_{i=1}^d(S^n_i\vee S^q_i)$. It is easy to check that the maps $\theta_D$ induce automorphisms $(\theta_D)_*:H_*(S^k\vee S^k)\to H_*(S^k\vee S^k)$ and the set of induced maps $(\theta_D)_*$ is isomorphic to $ GL_2(\mathbb Z)$.  Therefore, if $k\geq 2$, the maps $\theta_D$ are self-equivalences.

More generally, given any matrix $D\in GL_r(\mathbb Z)$ $(r\geq2)$, we define $\theta_D:\bigvee_{i=1}^rS^n_i\to\bigvee_{i=1}^rS^n_i$ as in \eqref{map:tdeg} such that the set of induced maps $(\theta_D)_*$ is isomorphic to $GL_r(\mathbb Z)$. Observe that we can extend this construction to the wedge sum of spheres in different dimensions $\bigvee_{i=1}^r (S^n_i\vee S^q_i)$. Thus given two matrices ${D_1}, {D_2}\in GL_r(\mathbb Z)$ defining homotopy equivalences $\theta_{D_1}:\bigvee_{i=1}^r S^n_i\to \bigvee_{i=1}^r S^n_i$ and $\theta_{D_2}:\bigvee_{i=1}^r S^q_i\to \bigvee_{i=1}^r S^q_i$ there is a map $\theta_{D_1,D_2}:\bigvee_{i=1}^r (S^n_i\vee S^q_i)\to \bigvee_{i=1}^r (S^n_i\vee S^q_i)$, and the set of induced maps $(\theta_{D_1,D_2})_*: H_k(\bigvee_{i=1}^r (S^n_i\vee S^q_i))\to H_k(\bigvee_{i=1}^r (S^n_i\vee S^q_i))$ is isomorphic to $GL_r(\mathbb Z)$ for $k=n,q$.
\end{proof}

Now we will analyse the homotopy type of sphere bundles over spheres and their suspensions. Let $X$ be the total space of a $q$-sphere bundle over an $n$-sphere $S^q\to X\xrightarrow{\pi} S^n$  $(n,q\geq2)$ which admit a cross sections. There is a fibration sequence
\begin{equation}
\xymatrix{
SO(q)\ar[r]^-j&SO(q+1)\ar[r]^-p&S^q.
}
\end{equation}
and a diagram of homotopy groups
\begin{equation}\label{d:spherefib}
\xymatrix{
\pi_n(S^q)\ar[r]^-\partial\ar[rd]_-{w}&\pi_{n-1}(SO(q))\ar[r]^-{j_{*}}\ar[d]^-J&\pi_{n-1}(SO(q+1))\ar[d]^-J\\
&\pi_{n+q-1}(S^q)\ar[r]^-E&\pi_{n+q}(S^{q+1})
}
\end{equation}
where $E$ is the suspension homomorphism, $w(\alpha)=[\alpha,\iota_q]$ and $J$ is the $J$-homomorphism, which commutes up to sign \cite{JW1}. That is, $w=-J\circ\partial$ and $E\circ J=-J\circ j_*$. Let $\chi(X)\in\pi_{n-1}(SO(q+1))$ be the characteristic element of the sphere bundle $X$. Since $X$ has cross sections then $\chi(X)=j_*\xi$ for some  $\xi\in\pi_{n-1}(SO(q))$.
The attaching map $\varphi$ of $X$ is given by 
\begin{equation}
\varphi=\iota_q\circ\eta+[\iota_n,\iota_q]
\end{equation}
where $\eta=J(\xi)$ and $\iota_q,\iota_n$ are generators of $\pi_q(S^q)$ and $\pi_n(S^n)$, respectively. 

\begin{lemma}
Let $S^q\to X\xrightarrow{\pi} S^n$ be a $q$-sphere bundle over an $n$-sphere $(n,q>1)$ which admits a cross section and let $\chi(X)=i_*\xi$, $\xi\in\pi_{n-1}(SO(q))$, be its characteristic map. There is a homotopy equivalence
$$\Sigma X\simeq S^{n+1}\vee \Sigma Y,$$
where $Y$ is the homotopy cofibre of the map $\iota_q\circ\eta\in\pi_{n+q}(S^q)$.
\end{lemma}

\begin{proof}
The attaching map of the top cell induces the following homotopy commutative diagram of cofibrations
\begin{equation}
\xymatrix{
{*}\ar[r]\ar[d]&S^q\ar@{=}[r]\ar[d]&S^q\ar[d]^-{s}\\
S^{n+q-1}\ar[r]^-{\varphi}\ar@{=}[d]&S^n\vee S^q\ar[r]\ar[d]^-{pinch}& X\ar[d]^-{c}\\
S^{n+q-1}\ar[r]^-{\iota_q\circ \eta}&S^n\ar[r]^-{}& Y
}
\end{equation}
which defines the map $c$ and the space $Y$. Observe that $s$ has a section. Thus after suspending we obtain the result.  
\end{proof}
Let $D^k$ be a $k$-disc and $\mathrm{Int}D^k$ be its interior. Given two oreintable $k$-manifolds $M_1$ and $M_2$ we can construct its connected sum $M_1\sharp M_2$ by taking $M_1-\mathrm{Int} D^k$ and $M_1-\mathrm{Int} D^k$ and identifying its boundaries. There are homotopy equivalences $M_1^{k-1}\simeq M_1-\mathrm{Int} D^k$ and $M_2^{k-1}\simeq M_2-\mathrm{Int} D^k$.The manifold $M_1\sharp M_2$ $$M_1\sharp M_2\simeq (M^{k-1}_1\vee M^{k-1}_2)\cup_f D^k,$$
where $f$ is the attaching map of the top cell. 
Let $A_i$, $0\leq i\leq r$, be total spaces of $n$-sphere bundles over $q$-spheres which admit cross sections.  Let $M=\sharp_{i=1}^r A_i$ be the connected sum $A_1\sharp\cdots\sharp A_r$ of the manifolds $A_i$. In \cite{Ish} Ishimoto classified connected sums of sphere bunlndles over spheres up to homotopy by studying the attaching maps of the top cell.

\begin{theorem}\cite[Lemma 1.1]{Ish}\label{ish}
Let $M_i$, $i=1,\dots,r$ be the total spaces of $S^q$-bundles over $S^n$ $(n,q\geq2)$ with characteristic elements $\chi(A_i)=i_*\xi^i$ for some given $\xi^i\in\pi_{n-1}(SO(q))$, $i=1,\dots,r$. Then $\sharp_{i=1}^rM_i$ has the homotopy type of $\{\bigvee_{i=1}^r(S_i^q\vee S^n_i)\}\cup_\varphi D^{q+n},$
and the homotopy class of the attaching map $\varphi$ is given by $$\varphi=\sum_{i=1}^r(\bar\eta^i+[\iota_n^i,\iota_q^i]),$$
where $\iota_q^i,\iota_q^n$ are orientation generators of $\pi_n(S^q_i)$, $\pi_n(S_i^n)$ respectively, $\bar\eta^i$ is given by the composite $ S^{n+q-1}\xrightarrow{J\xi^i} S^q\hookrightarrow S^n_i\vee S^q_i,$ $i=1,\dots,r$ and $\pi_{n+q-1}(S^q_i\vee S^n_i)$, $i=1,\dots,d$, are considered direct summands of $\pi_{n+q-1}(\vee_{i=1}^r(S^q_i\vee S^n_i))$.
\end{theorem}

Let $R_j^1=\bigoplus_{i=1}^r\pi_{n+q}(S^{n+1})$ and $R_j^2=\bigoplus_{i=1}^r\pi_{n+q}(S^{q+1})$ with minimal generating sets $\mathcal S_1$ and $\mathcal S_2$, respectively.  Let $t_1=|\mathcal S_1|$ and Let $t_2=|\mathcal S_2|$
\begin{lemma}\label{l:attachingcw}
 Let $X$ be a simply connected $CW$-complex such that $X=\left(\bigvee_{i=1}^r (S^n_i\vee S^q_i)\right)\cup D^{n+q}.$ There is a homotopy equivalence $\Sigma X\xrightarrow\simeq C$, where $C$ is the homotopy cofibre of a map $$g:S^{n+q}\to \Sigma X^{n+q-1}$$ represented by a matrix $$B_g=\begin{pmatrix} 
B_1 & 0 \\
0 & B_2
\end{pmatrix}.$$ Here $B_l$, $l=1,2$, are $r\times t_l$ matrices with row entries $b_{ij}$ given by the elements $\sum_{j=1}^{t_l}b_{ij}^1 \in R_j^1$ and $\sum_{j=1}^{t_2}b_{ij}^2\in R_j^2$, and each $B_l$ is in a row echelon form. \end{lemma}
\begin{proof}
Let $f:S^{n+q-1}\to X^{n+q-1}=\bigvee_{i=1}^r (S^n_i\vee S^q_i)$ be the attaching map of the top dimensional cell of $X$. After suspension we have that $\Sigma X^{n+q-1}= \bigvee_{i=1}^r (S^{n+1}_i\vee S^{q+1}_i)$ and  $\Sigma f\in\pi_{n+q}(\Sigma X^{n+q-1})$. By the Hilton-Milnor theorem it follows that $\pi_{n+q}( \bigvee_{i=1}^r (S^{n+1}_i\vee S^{q+1}_i))$ is a finitely generated abelian group. In particular we have that 
$$\pi_{n+q}(\Sigma X^{n+q-1})\cong\bigoplus_{i=1}^r\pi_{n+q}(S^{n+1})\oplus\bigoplus_{i=1}^r\pi_{n+q}(S^{q+1})\oplus WP,$$
where $WP$ is a subgroup generated by Whitehead products. Since Whitehead products vanish after suspension we have that $p_{WP}\circ\Sigma f\simeq *$ and $p\circ\Sigma f=\Sigma f_1+\Sigma f_2 \in R_1\oplus R_2\leq \pi_{n+q}(\Sigma X^{n+q-1})$, $R_1\cong\bigoplus_{i=1}^r\pi_{n+q}(S^{n+1})$, $R_2\cong\bigoplus_{i=1}^r\pi_{n+q}(S^{q+1})$. Here $p_{WP}$ and $p$ are the projections onto $WP$ and $R_1\oplus R_2$, respectively, and the groups $R_l$, $l=1,2$, are generated by the maps  $$S^{n+q}\xrightarrow{\alpha_{ij}^l} S^{k}_i\hookrightarrow S^{n+1}_i\vee S^{q+1}_i{},$$ where $i=1,\dots,r$, $j=1,\dots, t_l$ and $t_1=|\mathcal S_{n+1}|_{min}$, $t_2=|\mathcal S_{q+1}|_{min}$, and $|\mathcal S_k|_{min}$ is the minimal cardinality of a generating set of $\pi_{n+q}(S^k)$, $k=n+1,q+1$.

Since $G_1$ and $G_2$ can be regarded as $\mathbb Z$-modules, we can represent the maps $p\circ\Sigma f_l$, $l=1,2$,   as $r\times t_l$ matrices $A_l=(a_{ij}^l)$ with entries in $R_j^l=\mathbb Z$ within the columns $j=1,\dots t'_l$ and entries in the ring $R_j^l=\mathbb Z/s_j^l \mathbb Z$ within the column $j=t_l+i$, $i=1,\dots, t_l-t'_l$, $t_l'\in\mathbb N$. Thus we can regard $p\circ\Sigma f$ as a matrix 
\begin{equation}\label{eq:matrixA}
A=\begin{pmatrix} 
A_1 & 0 \\
0 & A_2 
\end{pmatrix}.
\end{equation} By Lemma \ref{l:matrixech} there are matrices $D_1,D_2\in GL_r(\mathbb Z)$ such that $B_1=D_1A_1$ and $B_2=D_2A_2$ are in a row echelon form.   Let $g\in G_1\oplus G_2$ be the map represented by a matrix $B$ obtained by replacing $A_l$ by $B_l$, $l=1,2$, in \eqref{eq:matrixA}. Since $D_1,D_2\in GL_r(\mathbb Z)$, by Lemma \ref{l:selfeqwedge} we 
can construct the self-equivalences $\theta_{D_1,D_2}:\bigvee_{i=1}^r (S^{n+1}_i\vee S^{q+1}_i)\to \bigvee_{i=1}^r (S^{n+1}_i\vee S^{q+1}_i)$ which induces isomorphisms in homology.  We obtain a homotopy commutative diagram 

\begin{equation}
\xymatrix{
S^{n+q}\ar[r]^-{\Sigma f}\ar@{=}[d]&\Sigma X^{n+q-1}\ar[r]^-a\ar[d]^{\theta_{D_1,D_2}}&\Sigma X\ar[d]^-{h}\\
S^{n+q}\ar[r]^-{g}&\Sigma X^{n+q-1}\ar[r]^-{a'}&C
}
\end{equation}
which defines the map $h$.  By the 5-lemma, $h$ induces isomorphisms in homology. Since all spaces are simply connected $CW$-complexes $h$ is a homotopy equivalence.

\end{proof}

Let $M=\sharp _{i=1}^r M_i$, where each summand $M_i$ is an $S^q$-bundle over $S^n$ with cross sections. Then by Lemma \ref{ish} the manifold $M$ has a cellular structure given by $(\bigvee^r_{i=1}S^q_i\vee\bigvee^r_{i=1}S^n_i)\cup_\Phi e^{n+q}$. Let $p: M \to \bigvee_{i=1}^d S_i^n$ be the composite
\begin{equation}
p: M \xrightarrow{\mathrm{pinch}} \bigvee^r_{i=1} M_{i} \xrightarrow{\bigvee^r_{i=1} \pi_i} \bigvee^r_{i=1} S_i^n,
\end{equation}
where the maps $\pi_i: M_{i} \to S_i^n$, $1 \leq i\leq r$, are bundle projections. 

For all $ 1\leq i\leq r$, let $s_i:S^n_i\to M$ be a cross section of the bundle $M_i\to S^n_i$ and let $g:\bigvee^r_{i=1} S_i^q\vee\bigvee^r_{i=1} S_i^n  \rightarrow  M$ be defined so that $S^q_i\hookrightarrow \bigvee^r_{i=1} S_i^q\vee\bigvee^r_{i=1} S_i^n  \xrightarrow g M$ is the inclusion of $S^q_i$ and $S^n_i\hookrightarrow \bigvee^r_{i=1} S_i^q\vee\bigvee^r_{i=1} S_i^n  \xrightarrow g M$ is the map $s_i$.
Consider the diagram
\begin{equation} \label{e:pipinch}
\xymatrix{
\bigvee^r_{i=1} S_i^n \ar@{=}[rrr] \ar@{^{(}->}[d] & &&\bigvee^r_{i=1} S_i^n \\
\bigvee^r_{i=1} S_i^q\vee\bigvee^r_{i=1} S_i^n  \ar[rrr]^-{\bigvee_{i=1}^rj_i\vee\bigvee _{i=1}^rs_i} \ar[d]^-g&&& \bigvee^r_{i=1} A_{i} \ar[u]_-{\bigvee^r_{i=1} \pi_i} \\
M \ar@{=}[rrr] &&& M. \ar[u]_-{\text{pinch}}
}
\end{equation}
where $j_i:S^q_i\to M_i$ is the inclusion. The bottom square homotopy commutes by definition of the map $g$ and the $CW$-structure of $M$, and the top square homotopy commutes since the composite
\begin{equation*}
 S_i^n \hookrightarrow  S_i^q \vee S_i^n  \xrightarrow {j_i\vee s_i} M_{i} \xrightarrow{\pi_i} S^n
\end{equation*}
is homotopic to the identity map. 

Given a family of $q$-sphere bundles over $n$-spheres $\mathcal F=\{A_i\}_{i=1}^r$ with characteristic maps $\chi(A_i)=i_*\xi^i$, $\xi^i\in\pi_q(SO(q))$, let $N_{\mathcal F}$ be the matrix formed with vector rows given by the elements $E J(\xi^{i})\in \pi_{n+q}(S^{q+1})$ and let $\bar  N_{\mathcal F}$ be its row echelon form. Let $rk(N_{\mathcal F})$ be the number of non-zero rows of $\bar N_{\mathcal F}$.

\begin{proposition}\label{p:Msuspension} 
Let $\mathcal F=\{M_i\}_{i=1}^r$ $(r\geq2)$ be a family of $q$-sphere bundles over $n$-spheres $(n,q\geq2)$ with characteristic maps $\chi(M_i)=i_*\xi^i$, for some $\xi^i\in\pi_q(SO(q))$ and let $M=\sharp_{i=1}^r M_i$ be a connected sum. Then there is a homotopy equivalence
 $$\Sigma M\simeq \bigvee_{i=1}^rS^{n+1}_i\vee \bigvee_{j=1}^{r-\bar t}S^{q+1}_{j}\vee \Sigma Y_{\mathcal F},$$
where $\bar t=min\{r,rk (\bar N_{\mathcal F})\}$  and $Y_{\mathcal F}$ is the cofibre of a map $$S^{n+q-1}\xrightarrow{B_{\mathcal F}}\bigvee^r_{i=1}S^q_i\xrightarrow{pinch}\bigvee_{i=1}^{\bar t}S^q_i.$$
 
\end{proposition}
\begin{proof}
The suspension of the attaching map of the top cell in $M$ induces a diagram of cofibrations 
\begin{equation}
\xymatrix{
{*}\ar[r]\ar[d]&\bigvee_{i=1}^r S^{n+1}_i\ar@{=}[r]\ar[d]&\bigvee_{i=1}^r S^{n+1}_i\ar[d]^-{s}\\
S^{n+q}\ar[r]^-{\varphi}\ar@{=}[d]&\bigvee_{i=1}^r(S_i^{q+1}\vee S^{n+1}_i)\ar[r]\ar[d]^-{pinch}& \Sigma M\ar[d]^-{c}\\
S^{n+q}\ar[r]^-{g}&\bigvee^{r}_{i=1}S^{q+1}_{i}\ar[r]^-{}&C
}
\end{equation}
which homotopy commutes. The map $\Sigma s$ has a left homotopy inverse, namely, the composite
\begin{equation}\label{eq:Proj}
 \Sigma M \xrightarrow{\mathrm{pinch}} \bigvee^r_{i=1} \Sigma M_{i} \xrightarrow{\bigvee^r_{i=1} \pi_i} \bigvee^r_{i=1} S_i^{n+1}.
\end{equation}
Therefore $\Sigma M\simeq \bigvee_{i=1}^r S^{n+1}_i\vee C$. By Lemma \ref{l:attachingcw} there is a cofibration sequence $$S^{n+q}\xrightarrow{B_{\mathcal F}} \bigvee_{i=1}^{r} S^{q+1}_i\to \bigvee^{r-\bar t}_{i=1}S^{q+1}_i\vee\Sigma Y_{\mathcal F},$$
and a homotopy equivalence $C\simeq  \bigvee^{r-\bar t}_{i=1}S^{q+1}_i\vee\Sigma Y_{\mathcal F}$.

\end{proof}
  
\section{Principal $G$-bundles over $M$} \label{s:prin-G}

Let $X$ and $Y$ be $CW$-complexes and let Map$(X,Y)$ and Map$_*(X,Y)$ be the mapping spaces of unpointed and pointed maps from $X$ to $Y,$ respectively, endowed with the compact-open topology. We denote the connected components containing a map $f$ by Map$^f(X,Y)$ and Map$^f_*(X,Y)$. Let $[X,Y]$ and $[X,Y]_*$ be the sets of homotopy classes of unpointed and pointed maps from $X$ to $Y$, respectively. It is a standard result that the isomorphism classes of  principal $G$-bundles over a CW-complex $X$ are in one-to-one correspondence with $[X,BG]$. Therefore we compute the sets $[M,BG]$ for manifolds $M$ that are connected sums of  $n$-sphere bundles over a $q$-spheres. From the evaluation fibration
\begin{equation*}
{\rm{Map}}_*(M,BG)\rightarrow{\rm{Map}}(M,BG)\xrightarrow{ev}BG
\end{equation*}
we obtain an exact sequence of homotopy sets
$$\pi_1(BG)\xrightarrow{\partial} [M,BG]_*\to[M,BG]\xrightarrow{ev^*}\pi_0(BG).$$
The map $ev^*$ is trivial since $BG$ is connected. Since all the groups $G$ considered in this work are simply connected, $\partial=0$ and the action of $\pi_1(BG)$ on $[M,BG]_*$ is trivial, which implies that there is a bijection between $[M,BG]$ and  $[M,BG]_*$. 

\begin{lemma}\label{l:classext}

Let $r \geq 2$. Let $\mathcal F=\{M_i\}^r_{i=1}$ be a family of $S^q$-bundles over $S^n$ with classifying maps $\chi(M_i)=j_*\xi_i$, $\xi_i\in\pi_{n-1}(SO(q))$ and let $M=\sharp_{i=1}^rM_i$ be a  connected sum. Suppose that $G$ is one of the following groups: 
\begin{enumerate}
\item $SU(m)$, $2m\geq n+q$;
\item $Sp(m)$, $4m\geq n+q-2$.
\end{enumerate} 
Then $$[M,BG]=\bigoplus_{i=1}^r\{\pi_{n-1}(G)\}_i\oplus\bigoplus_{j=1}^{r-rk(B)}\{\pi_{q-1}(G)\}_j\oplus[Y_{\mathcal F}, BG],$$
where $Y_\mathcal F$ is the homotopy cofibre of a map $\bigvee_{i=1}^r S^{n+1}_i\vee \bigvee_{j=1}^{r-\bar t}S^{q+1}_j\hookrightarrow \Sigma M$, for some $\bar t\in\mathbb N$.
\end{lemma}
\begin{proof}
If $G=SU(r)$ with $2r> n+q$ then $[M,BSU(r)]=[M, BSU(\infty)]$. Similarly, If $G=Sp(r)$ with $4r> n+q-2$ then $[M,BSp(r)]=[M, BSp(\infty)]$. Since $BSU(\infty)$ and $BSU(\infty)$ are infinite loop spaces, we can suspend $M$ and use Proposition \ref{p:Msuspension} to obtain the result.
\end{proof}

Proposition \ref{l:classext} is a general result that allows to classify principal $G$-bundles over connected sums $M$. We will restrict from now on to the case of odd-dimensional connected sums of $S^q$-bundles over $S^n$ for which $n=2k$ and $q=2k'-1$ for some $2\leq k'\leq k$. The following technical result will be required in the computation of $[M,BG]$. Let  \begin{equation}\label{c:s3inc}
\xymatrix{
\bigvee^r_{i=1}S^q_i\ar[r]^-{j}& M\ar[r]^-{q}& W.
}
\end{equation} 
be the cofibration sequence associated to the inclusion of the $q$-skeleton of $M$.

\begin{lemma}\label{quotient}
Let $M=\sharp _{i=1}^r M_i$, where each summand $M_i$ is an $S^q$-bundle over $S^n$ and $2\leq q<n$. There is a homotopy equivalence $$W\simeq (\bigvee^r_{i=1}S^{n}_i)\vee S^{n+q}$$ and the homotopy equivalence can be chosen so that the composite $$M\xrightarrow{q} W\xrightarrow{\simeq}(\bigvee^r_{i=1}S^{n}_i)\vee S^{n+q}\xrightarrow{pinch}\bigvee^r_{i=1} S^n_i$$
is homotopic to the map $p$.
\end{lemma}
\begin{proof}
Since $q<n$, $\pi_q(S^n)\cong0$, theres is a homotopy extension of $p:M\rightarrow \bigvee_{i=1}^rS^n$
\begin{equation}\label{C2}
\xymatrix{
\bigvee^r_{i=1} S^q\ar[r]^-{j}&M\ar[r]^-q\ar[d]_-{p}&W_{}\ar@{.>}[ld]^-{\tilde p}\\
&\bigvee^r_{i=1}S^n.
}
\end{equation}
The cofibre $W$ can be constructed as a $CW$-complex with one $(n+q)$-cell attached to a wedge sum of $n$-spheres. The attaching map of the top cell in $W$ induces a cofibration sequence
\begin{equation}\label{cwc}
\xymatrix{
S^{q+n-1}\ar[r]^{\varphi}& \bigvee^r_{i=1} S^n\ar[r]^-{h}&W\ar[r]^-{b}&S^{q+n},
}
\end{equation}
where $h$ is the inclusion and $b$ is the connecting map. The map $p:M\to \bigvee^r_{i=1}S^n_i$ has a right homotopy inverse, namely the composite $$ \bigvee^r_{i=1} S^n_i\hookrightarrow \bigvee^r_{i=1} S^n_i\vee\bigvee^r_{i=1} S^q_i \xrightarrow i M,$$
and by the homotopy commutativity of \eqref{C2}, and so does the map $\tilde p$. Hence the composite $$\bigvee^r_{i=1} S^n_i\xrightarrow{h} W\overset{\tilde p}{\to}\bigvee^r_{i=1}S^n_i$$ is a homotopy equivalence. There exists a coaction $\sigma:W\to W\vee S^{q+n}$ so that the composition 
$W\xrightarrow{\sigma} W\vee S^{q+n}\xrightarrow{pinch}S^{q+n}$
is homotopic to the connecting map $b$ in the cofibration sequence \eqref{cwc} and the composite
$$W\xrightarrow{\sigma} W\vee S^{q+n}\xrightarrow{pinch}W$$ is a homotopy equivalence. Therefore, since $\tilde p$ has a right homotopy inverse, the map
\begin{equation}
\theta:W\xrightarrow{\sigma}W\vee S^{q+n}\xrightarrow{\tilde p\vee\mathbbm 1}\bigvee_{i=1}^rS^n_i\vee S^{q+n}
\end{equation}
is a homotopy equivalence. Observe that $\theta$ is homotopic to $\tilde p$. Therefore, by the homotopy commutativity of \eqref{C2}, the composite $$M\xrightarrow{q}W\xrightarrow{\theta}\bigvee_{i=1}^rS^n\vee S^{q+n}\xrightarrow{pinch}\bigvee^r_{i=1}S^n$$ is homotopic to $p$.
\end{proof}

\begin{proposition}\label{p:pb_gen}
Let $k,k',m\in\mathbb N$ be such that $k+k'\leq m$, let $G$ be a Lie group and let $\mathcal F=\{M_i\}_{i=1}^r$ be a family of manifolds that arise as total spaces of $S^{q}$-bundles over $S^{n}$ with characteristic elements $\chi(M_i)=j_*\xi^i$, $\xi^i\in\pi_{n-1}(SO(q))$. Let $M=\sharp_{i=1}^rM_i$, $r\geq1$, be a connected sum. Suppose one of the following holds 
\begin{enumerate}
\item $G=SU(m)$, $n=2k$, $q=2k'-1$, $2 \leq k' \leq k$;
\item $G=Sp(m)$, $n=4k$, $q=4k'-1$, $1 \leq k' \leq k$.
\end{enumerate}
Then the map $$p:M_{}\to \bigvee_{i=1}^r M_i\xrightarrow{\bigvee\pi_i} \bigvee_{i=1}^r S^{n}$$ induces a bijection $p^*:\oplus_{i=1}^r\pi_{n-1}(G)\to [M,BG]$ and there is a one-to-one correspondence $$[M,BG]=\mathbb Z^r.$$ 
\end{proposition}

\begin{proof}
By Lemma \ref{quotient} there is a homotopy cofibration sequence
 \begin{equation}\label{iso}
\xymatrix{
\bigvee^r_{i=1}S^q\ar[r]^-{i}&M\ar[r]^-{a}& \bigvee^r_{i=1}S^n\vee S^{n+q}\ar[r]^-{\delta}&\bigvee^r_{i=1}S^{q+1}
}
\end{equation}
where $i$ is the inclusion, $a$ is the composite $M\xrightarrow{q}W\xrightarrow{\simeq}\bigvee^r_{i=1}S^n_i\vee S^{n+q}$ and $\delta$ is the connecting map. Let $M^{n}$ be the $n$-skeleton of $M$. Then $M^n\simeq \bigvee^r_{i=1}(S^q_i\vee S^n_i)$ which is a co-$H$-space. From the exact sequence induced by the attaching map of the $n$-cells,
\begin{equation}\label{c:4sk}
\xymatrix{
\bigvee^r_{i=1}S^{n-1}_i\ar[r]^-{*}&\bigvee^r_{i=1}S^q_i\ar[r]&M^n\ar[r]&\bigvee^r_{i=1}S^{n}_i,}
\end{equation}
we obtain an exact sequence of groups
\begin{equation}\label{n-skcof}
\xymatrix{
[\bigvee^r_{i=1}S^{q+1},BG]\ar[r]^-{0^*}\ar@{=}[d]&[\bigvee^r_{i=1}S^n_i,BG]\ar[r]\ar@{=}[d]&[M^n,BG]\ar[r]&[\bigvee^r_{i=1}S^{q},BG]\ar@{=}[d]\\
\oplus_{i=1}^r\pi_{q+1}(BG)&\oplus_{i=1}^r\pi_n(BG)&&\oplus_{i=1}^r\pi_{q}(BG)
}
\end{equation}
where $0^*$ is the trivial map. Suppose that one of the conditions of the proposition holds. Then there are isomorphisms \cite{Jam2} 
\begin{equation}\label{isoLG}
\oplus_{i=1}^r\pi_q(BG)\cong\oplus_{i=1}^r\pi_{q-1}(G)\cong0.
\end{equation}
\begin{equation}
\oplus_{i=1}^r\pi_n(BG)\cong\oplus_{i=1}^r\pi_{n-1}(G)\cong\mathbb Z^r.
\end{equation}
Exactness of \eqref{n-skcof} implies that $[M^n,BG]=\mathbb Z^r$. The coaction $$\psi :\bigvee^r_{i=1}S^{n}\to \bigvee^r_{i=1}(S^{q+1}_i\vee S^n_i)$$ induces an action $\psi^*$
\begin{equation*}
\psi^*:[\bigvee^r_{i=1}S^n_i,BG]\times [\bigvee^r_{i=1}S^{q+1}_i,BG]\to [\bigvee^r_{i=1}S^n_i,BG].
\end{equation*}
The inclusion of the wedge $\bigvee_{i=1}^rS^q\xrightarrow i M$ factors through the $n$-skeleton $M^n$, which induces a homotopy commutative diagram of cofibrations
\begin{equation}\label{d:coactions}
\xymatrix{
\bigvee^r_{i=1}S^q_i\ar[r]\ar@{=}[d]&M^n\ar[r]\ar[d]&\bigvee^r_{i=1}S^n_i\ar[r]^-{m}\ar[d]_-{i_1}& \bigvee^r_{i=1}S^{q+1}_i\ar@{=}[d]\\
\bigvee^r_{i=1}S^q_i\ar[r]^-i&M\ar[r]^-{a}&(\bigvee^r_{i=1}S^n_i)\vee S^{n+q}\ar[r]^-\delta & \bigvee^r_{i=1}S^{q+1}_i
}
\end{equation}
where $i_1:\bigvee^r_{i=1}S^n_i\to (\bigvee^r_{i=1}S^n_i)\vee S^{n+q}$ is the inclusion into the wedge.  From \eqref{d:coactions} 
we obtain a homotopy commutative diagram as follows
\begin{equation}\label{d:coactions2}
\xymatrix{
\bigvee^r_{i=1}S^n_i\ar[r]^-{\psi}\ar[d]_-{i_1}& \bigvee^r_{i=1}S^n_i\vee \bigvee^r_{i=1}S^{q+1}_i\ar[d]^-{i_1\vee\mathbbm{1}}\\
\bigvee^r_{i=1}S^n_i\vee S^{n+q}\ar[r]^-{\psi'}&( \bigvee^r_{i=1}S^n_i)\vee S^{n+q}\vee \bigvee^r_{i=1}S^{q+1}_i,
}
\end{equation}
where $\psi'$ is the coaction associated to the bottom row in \eqref{d:coactions}. Applying the functor $[-,BG]$ we obtain a commutative diagram of homotopy groups
\begin{equation}\label{d:coactions3}
\xymatrix{
\oplus_{i=1}^r\pi_n(BG)\oplus \pi_{n+q}(BG)\oplus\oplus_{i=1}^r\pi_{q+1}(BG)\ar[r]^-{(\psi')^*}\ar[d]_-{i_1^*\times\mathbbm{1}}& \oplus_{i=1}^r\pi_n(BG)\oplus\pi_{n+q}(BG)\ar[d]^-{i_1^*}\\
 \oplus_{i=1}^r\pi_n(BG)\oplus\oplus_{i=1}^r\pi_{q+1}(BG)\ar[r]^-{\psi^*}& \oplus_{i=1}^r\pi_n(BG).\\
}
\end{equation}
Under the conditions on $G$ it follows that $\pi_{n+q}(BG)\cong\pi_{n+q-1}(G)\cong0$. The vertical arrows in \eqref{d:coactions3} are isomorphisms implying that $(\psi')^*=\psi^*$. Since $(\psi')^*=\psi^*$, the induced map $a^*:[\bigvee^r_{i=1}S^n_i\vee S^{n+q},BG]\to[M,BG]$ is an isomorphism and therefore $[M_{},BG]=\mathbb{Z}^r$.

By Lemma \ref{quotient} the composite $$M\xrightarrow{q} W\xrightarrow{\simeq}(\bigvee_{i=1}^rS^n_i)\vee S^{n+q}\xrightarrow{pinch}\bigvee_{i=1}^r S^{n}_i$$
is homotopic to the map $p:M\to \bigvee_{i=1}^rS^n_i$.
 Consider the commutative diagram
\begin{equation}\label{d:tri}
\xymatrix{
[M,BG]&[\bigvee^r_{i=1}S^n_i,BG]\times[S^{n+q},BG]\ar[l]_-{a^*} \\
[\bigvee^r_{i=1}S^n_i,BG]\ar[ur]_-{q_1}\ar[u]^-{p^{*}}
}
\end{equation}
where the map $q_1$ is the inclusion of the first factor. Since $a^*$ is an isomorphism, by the commutativity of \eqref{d:tri} it follows that the induced map $$p^{*}:[M,BG]\to[\bigvee^r_{i=1}S^n_i,BG]=\oplus_{i=1}^r\pi_{n-1}(G)\cong\mathbb Z^r$$ is also an isomorphism. 
\end{proof}

We want extend the classification result in \cite{IMS} for the principal $G$ bundles over $S^3$-bundles over $S^4$ with torsion-free homology to the case when the manifold is given as a  connected sum and $G$ is any simply connected simple compact Lie group. 

Let $M_{\xi}$ be the $S^3$-bundle over $S^4$ classified by $(\xi,0)\in\pi_3(SO(4))\cong\pi_3(SO(3))\oplus\pi_3(S^3)=\mathbb{Z}\oplus\mathbb{Z}$. Then the projection map $\pi:M_{\xi}\to S^4$ admits cross sections and we have the following homotopy equivalences
\begin{itemize}
\item $M_{\xi}\simeq M_{\xi'}$ if and only if $\xi\equiv \pm \xi' \pmod{12}$;
\item $M_{\xi}\simeq S^3\times S^4$ if and only if $\xi\equiv 0 \pmod{12}$.
\end{itemize}
Let $|\pi_6(G)|$ be the order of $\pi_6(G)$. In the homotopy theory of gauge groups of principal $G$-bundles over connected sums of $S^3$-bundle over $S^4$, the values of $\xi$ play a key role if $|\pi_6(G)|> 1$. 

Let $M=\sharp_{i=1}^rM_{\xi^i}$, where each factor $M_{\xi^i}$ is the total space of an $S^3$-bundle over $S^4$ classified by $\xi^i\in\pi_3(SO(3))$. The manifold $M$ is 2-connected and has torsion-free homology. The attaching map of the top cell in the $CW$-structure of $M$ induces a homotopy cofibration sequence 
\begin{equation}
\xymatrix{ 
S^6 \ar[r]^-{\varphi}&\bigvee_{i=1}^r (S^3_i\vee S^4_i)\ar[r]^-i&M\ar[r]^q & S^{7}\ar[r]^-{\Sigma \varphi}& \bigvee_{i=1}^r (S^4_i\vee S^5_i),
}
\end{equation}
where $\varphi$ is the attaching map of the top cell, $i$ is the inclusion of the 4-skeleton and $q$ is the pinch map to the top cell. Recall that the attaching can be expressed as $\varphi=\sum_{i=1}^r(\bar\eta^i+[\iota_n^i,\iota_q^i]),$ where $\iota_3^i,\iota_4^i$ are generators of $\pi_4(S^3_i)$, $\pi_4(S_i^4)$ respectively and $\bar\eta^i$, $ 1\leq i\leq r$, is given by the composite $$S^{6}\xrightarrow{J(\xi^i)} S^3_i\hookrightarrow S^4_i\vee S^3_i.$$
We obtain an exact sequence
\begin{equation}\label{cofprods1}
\xymatrix{ 
[\bigvee_{i=1}^r (S^4_i\vee S^5_i),BG]\ar[r]^-{(\Sigma \varphi)^{*}}&[S^7,BG]\ar[r]^-{q^{*}}& [M,BG]\ar[r]^-{i^{*}}&[\bigvee_{i=1}^r (S^3_i\vee S^4_i),BG]\ar[r]^-{ \varphi^{*}}& [S^6,BG].
}
\end{equation}
A generator of $\pi_6(S^3) \cong \mathbb{Z}_{12}$ is given by the Samelson product $\langle \iota,\iota \rangle$. The $J$-homomorphism $J: \pi_3(SO(3))\to\pi_6(S^3)$ sending $\xi^i$ to $\xi^i \mod 12$, is an epimorphism and $\mathrm{Im}(E\circ J)=\mathbb Z_{12}$ is generated by $\Sigma\langle\iota,\iota\rangle$. It follows from Lemma \ref{l:attachingcw} that there is a homotopy $p_1\circ\Sigma \varphi\simeq \ell\Sigma\mathcal\langle \iota,\iota \rangle$, where $\ell=\gcd(12,\xi^1,\dots, \xi^r)$ and $p_1$ is the pinch onto $S^3_1$. 

\begin{lemma}\label{l:coker}
Let $G$ be a simply connected simple compact Lie group. Suppose that we have $\gcd(|\pi_6(G)|,\xi^1,\dots,\xi^r) = 1$. Then $(\Sigma\varphi)^*$ is surjective; in particular, the image of $(\Sigma\varphi)^*$ is isomorphic to $\mathbb Z_{|\pi_6(G)|}$. % where $n={\frac{|\pi_6(G)|}{\gcd(|\pi_6(G)|,\xi^1,\dots,\xi^r)}}$
\end{lemma}
\begin{proof}
First notice that if $|\pi_6(G)|=1$ then $(\Sigma\varphi)^*:\oplus_{i=1}^r\pi_3(G)\oplus\oplus_{i=1}^r\pi_5(G)\to \pi_6(G)$ is the trivial map. Therefore $\
\mathrm{Im}(\Sigma\varphi)^*\cong0\cong\mathbb Z_1$. %$\cong\mathbb {Z}_{\frac{1}{\gcd(1,\xi^1,\dots,\xi^r)}}$.

Now suppose $|\pi_6(G)|>1$. Then $G$ is $SU(2)\cong Sp(1)$, $SU(3)$ or $G_2$. %If $\xi^i\equiv0 \pmod{12}$ for all $1\leq i\leq r$, then $M_{\xi^i}\simeq S^3\times S^4$. In this case the attaching map $\varphi$ is a linear combination of Whitehead products. Hence $\Sigma\varphi$ is nullhomotopic, implying that $(\Sigma\varphi)^*$ is the trivial map. Thus $\mathrm{Im}(\Sigma\varphi)^*\cong\mathbb {Z}_{\frac{|\pi_6(G)|}{\gcd(|\pi_6(G)|,\xi^1,\dots,\xi^r)}}\cong0$.
%Finally, suppose $\xi^i\not\equiv 0\pmod{12}$ for some $1\leq i\leq r$. Then f
From the exact sequence \eqref{cofprods1} we have $\mathrm{Im}(\Sigma \varphi)^*=\ker q^*$. A map $f:S^7\to BG$ is in the kernel of $q^*$ if and only if there is an extension $\tilde f:\bigvee_{i=1}^rS^4_i\vee S^5_i\to BG$, such that the diagram
\begin{equation}
\xymatrix{
M\ar[r]^-{q}&S^7\ar[r]^-{\Sigma\varphi}\ar[d]_{f}& \bigvee_{i=1}^r (S^4_i\vee S^5_i)\ar@{.>}[dl]^-{\tilde f}\\
&BG
}
\end{equation}
 homotopy commutes. We claim that the generator $\tilde\gamma$ of $\pi_6(G)$ factors as $\tilde\gamma:S^6\xrightarrow {\langle\iota,\iota\rangle}S^3\hookrightarrow G$. In this case, the adjoint of $\tilde\gamma$, $$\tilde\gamma^{\rm{Ad}}:S^7\xrightarrow{\Sigma \langle\iota,\iota\rangle}S^4\hookrightarrow BG,$$ is a  generator of $\pi_7(BG)$.
 
For $G=SU(2)$, our claim is trivial. For $G = SU(3)$, consider the exact sequence of homotopy groups induced by the fibre bundle $S^3\to SU(3)\to S^5$,
\begin{equation}\label{f:su3}
\pi_6(S^3)\xrightarrow{i_1^*} \pi_6(SU(3))\xrightarrow{p_1^*} \pi_6(S^5)\xrightarrow{\delta^*} \pi_5(S^3)\to \pi_5(SU(3)).\end{equation}
 From \cite{Tod} we have $\pi_6(S^5)\cong\pi_5(S^3)\cong \mathbb Z_2$ and $\pi_5(SU(3))\cong \mathbb Z$. This implies that the last map in \eqref{f:su3} is the zero map. Therefore $\delta^*$ is surjective, and as $\delta^*$ is a map between two copies of $\mathbb{Z}_2$, it must be an isomorphism. In turn, $p_1^*$ is the zero map. Thus any map $S^6\to SU(3)$ factors as a composite $$S^6\to S^3\overset{i_1}{\to}SU(3),$$ proving our claim for $G=SU(3)$. For $G = G_2$, localising at $p=3$ there is an exact sequence \cite{MimJam}
 $$ \pi_6(S^3)\overset{i_2^*}{\to} \pi_6(G_2)\overset{p_2^*}{\to} \pi_6(S^{11}).$$
Since $\pi_6(S^{11})\cong0$, the map $i_2^*$ is surjective when localised at 3. Since $\pi_6(G_2) \cong \mathbb{Z}_3$ is invariant under localisation at 3, the map $i_2^*$ is surjective integrally. Thus any map $S^6\to G_2$ factors as a composite $$S^6\to S^3\xrightarrow{i_2}G_2,$$ proving our claim in the case $G=G_2$.

In the Table \ref{generators} we collect this information on the generators of the non-trivial groups $\pi_6(G)$, that is, when  $G=SU(2),SU(3)$ or $G_2$.
\begin{table}
\begin{center}
\begin{tabular}{|c|c|c|}
\hline
$G$&$\pi_6(G)$&generator\\
\hline
$SU(2)\cong S^3$&$\mathbb{Z}_{12}$&$\langle \iota,\iota\rangle$\\
$SU(3)$&$\mathbb{Z}_6$&$i_1\circ\langle\iota,\iota\rangle$\\
$G_2$&$\mathbb{Z}_3$&$i_2\circ \langle\iota,\iota\rangle$\\
\hline
\end {tabular}
\caption{Homotopy groups $\pi_6(G)$ and their generators }\label{generators}
\end{center}
\end{table}
We have that $p_1 \circ \Sigma\varphi\simeq \ell\Sigma \langle \iota,\iota\rangle $, where the map $p_1:\bigvee_{i=1}^r S^4_i \vee S^5_i \to S^4_1$ is the pinch onto $S^4_1$ and $\ell=\gcd(12,\xi^1,\dots,\xi^r)$. Therefore the diagram
\begin{equation}\label{coker}
\xymatrix@=0.4in{ 
M_{}\ar[r]^-{h}\ar[d]_-{q} &BG\\
S^7\ar[r]^{\ell k\Sigma \langle\iota,\iota\rangle}\ar[d]_-{\Sigma\varphi} &S^4\ar[u]^{i} \\
\bigvee_{i=1}^rS^4_i\vee S^5_i\ar[r]^-{p_1}&S^4_1,\ar[u]_k
} 
\end{equation}\\
where $k$ is the degree $k$ map, homotopy commutes. The maps $i\circ \ell k\Sigma\langle\iota,\iota\rangle$ are therefore in the kernel of $q^*.$ The result follows from the exact sequence \eqref{cofprods1}, the diagram \eqref{coker}, Table \ref{generators} and the fact that, by assumption, $\ell = 1$.
\end{proof}

\begin{lemma}\label{l:Phi0}
Let $M$ be a connected sum of sphere bundles $M_{\xi^i},$ for $1\leq i\leq d.$ The map $$\varphi^*: [M^4,BG]\to\pi_6(BG)$$ is trivial. 
\end{lemma}
\begin{proof}
Let $G$ be any of the groups stated. By connectivity we have $[M^4,BG]=[{\bigvee}_{i=1}^r S^4_i,BG]$. Therefore any map $f:M^4\to BG,$ factors as the composite $$M^4\xrightarrow{\rm{pinch}}\bigvee^r_{i=1} S_i^4\rightarrow  BG.$$ 
 Thus there is a commutative diagram
\begin{equation*}
\xymatrix@=0.5in{ 
S^6\ar[r]^{\tilde\varphi}\ar[d]_{{\varphi}} &{\bigvee}^r_{i=1}S^4_i\ar[d]^{j} \\
M^4\ar[r]^{}\ar[ur]^{ {pinch}} &BG.
} 
\end{equation*}
where $\tilde\varphi={{pinch}}\circ\varphi$. Thus $\tilde\varphi\in \pi_6(\bigvee_{i=1}^r S^4_i)\cong\oplus_{i=1}^r\pi_6(S^4_i)$. The composite $$S^6\xrightarrow{\varphi}\bigvee_{i=1}^rS^4_i\xrightarrow{pinch}S^4_i$$ is nullhomotopic for all $i$, since all manifolds $M_{\xi^i}$ have cross sections. Therefore $\tilde\varphi$ is nullhomotopic, implying that $\varphi^{*}$ is trivial.  
\end{proof}

    We present a classification of principal $G$-bundles which generalises Proposition 2.3 in \cite{IMS} fore the case of $S^3$-bundles over $S^4$ with torsion-free homology. That is, we include connected sums and the Lie groups $G$ such that $|\pi_6(G)|>1$, namely $SU(2)$, $SU(3)$ and $G_2$. These result also extends the computations of $[M,BG]$ given in Proposition \ref{p:pb_gen} for the case $n=4$ and $q=3$.

\begin{proposition}\label{classext}

Let $M=\sharp_{i=1}^r M_{\xi^i}$, $\xi^i\in\mathbb Z$, $r\geq 1$. If $G$ is a simply connected simple compact Lie group and $\gcd (|\pi_6(G)|,\xi^i,\dots,\xi^r)=1$, then $$[M,BG]=\mathbb{Z}^r.$$
Moreover, the map $$p:M\xrightarrow{pinch}\bigvee_{i=1}^rM_{\xi^i}\xrightarrow{\vee\pi_i} \bigvee_{i=1}^rS^4_i$$ induces a bijection $p^*:\oplus_{i=1}^r\pi_3(G)\to [M,BG]$.
\end{proposition}

\begin{proof}
Consider the following exact sequence
\begin{equation}\label{c}
\xymatrix{ 
[\bigvee_{i=1}^r(S^4_i\vee S^5_i),BG]\ar[r]^-{(\Sigma \varphi)^{*}}&[S^7,BG]\ar[r]^-{q^{*}}& [M,BG]\ar[r]^-{i^{*}}&[\bigvee_{i=1}^r(S^3_i\vee S^4_i),BG]\ar[r]^-{ \varphi^{*}}& [S^6,BG].
}
\end{equation}

First suppose $|\pi_6(G)|=1$. An argument along the lines of the proof of Proposition \ref{p:pb_gen} shows that the map $i^*$ is injective. By Lemma 
\ref{l:Phi0} the map $i^*$ is surjective. Therefore we obtain in this case $$[M,BG]=\oplus_{i=1}^r\pi_4(BG)\cong\mathbb{Z}^r.$$
Now suppose $|\pi_6(G)|=1$. Notice that even if $q^* = 0$ in \eqref{c} this does not imply immediately that $i^*$ is injective, since in general the set $[M,BG]$ is not a group.
For $j \in\mathbb{Z}^r$, let $\alpha_j \in  \oplus_{i=1}^r\pi_4(BG)$ be the map corresponding to $j$ under the isomorphism $ \oplus_{i=1}^r\pi_4(BG) \cong \mathbb{Z}^r$.
According to Theorem 3.2.1 in \cite{Rut}, we can define maps $\Gamma(\alpha_j,\varphi):[\bigvee_{i=1}^r(S^4_i\vee S^5_i),BG]\to [S^7,BG]$  for each $j\in\mathbb Z^r$ such that 
$$[M,BG]=\bigcup_{j\in\mathbb Z^r} \mathrm{coker}\Gamma(\alpha_j,\varphi).$$ 
Moreover, from Theorem 3.3.3 in \cite{Rut} we have that if $\varphi^*$ is a homomorphism then 
\begin{equation}\label{gamma}
\Gamma(\alpha_j,\varphi)=(\Sigma\varphi)^*.
\end{equation}
Since the map $\varphi^*$ is trivial, equality in \eqref{gamma} holds and we have 
\begin{equation}
[M,BG]=\mathbb Z^r\times \mathrm{coker}(\Sigma\varphi)^*,
\end{equation}
where $\mathrm{coker}(\Sigma\varphi)^*=\pi_6(G)/\mathrm{Im}(\Sigma\varphi)^*$.
Using Lemma \ref{l:coker}  we have $\mathrm{coker}(\Sigma\varphi)^*=0.$

Finally, the map $p:M\to\bigvee _{i=1}^rS^4_i$ has a right homotopy inverse so that the composite $$\bigvee _{i=1}^rS^4_i\hookrightarrow M\overset{p}{\to} \bigvee _{i=1}^rS^4_i$$ is a homotopy equivalence. Therefore the composite
\begin{equation}\label{e:bijection}
\bigvee _{i=1}^rS^4_i\hookrightarrow \bigvee _{i=1}^r(S^3_i\vee S^4_i)\xrightarrow i M\overset{p}{\to} \bigvee _{i=1}^rS^4_i
\end{equation}
is a homotopy equivalence.  Thus applying the functor $[-,BG]$ to \eqref{e:bijection} shows that the map $$p^*:\oplus_{r=1}^r\pi_4(BG)\to[M,BG]$$ is a bijection.
\end{proof}
\begin{remark}
Proposition \ref{classext}  allows to recover the conclusion of Proposition 2.3 in \cite{IMS} for the case $r=1$.
\end{remark}

\begin{cor}\label{cor:pringb}
Let $k,k',m\in\mathbb N$ be such that $k+k'\leq m$, let $G$ be a Lie group and $\mathcal F=\{M_i\}_{i=1}^r$ be a family of manifolds that arise as total spaces of $S^{q}$-bundles over $S^{n}$ with characteristic elements $\chi(M_i)=j_*\xi^i$, $\xi^i\in\pi_{n-1}(SO(q))$. Let $M=\sharp M_i^r$ be a connected sum. Suppose that one of the following holds:
\begin{enumerate}
\item $G=SU(m)$, $n=2k$, $q=2k'-1$, $2 \leq k' \leq k$;
\item $G=Sp(m)$, $n=4k$, $q=4k'-1$, $1 \leq k' \leq k$;
\item $n=4$, $q=3$, $\gcd (|\pi_6(G)|,\xi^1,\dots,\xi^r)=1$ and $G$ is a simply connected simple compact Lie group.
\end{enumerate}
Then there is a one-to-one correspondence $Prin_G(M)\xrightarrow{1-1}\mathbb Z^r$.
\end{cor}

\section{Homotopy decompositions of gauge groups over connected sums} \label{s:gauge-groups}

In this section we explore the homotopy theory of both pointed and unpointed gauge groups of principal $G$-bundles over connected sums. Let $P_f\rightarrow X$ be a principal $G$-bundle classified by a map $f:X\to BG$. The unpointed gauge group of the bundle, denoted $\mathcal G^f(X)$, is the group of its bundle automorphisms. Thus an element $\phi\in \mathcal G^f(X)$ is  $G$-equivariant automorphism of $P_f$ and covers the identity on $X$. The pointed gauge group $\mathcal G_*^f(X)$ is the subgroup of $\mathcal{G}_f(X)$ that fixes the fibre at the basepoint pointwise. Let $B\mathcal G^f(X)$ be the classifying space of $\mathcal G^f(X)$. 
By \cite{Got} or  \cite{AB} there are homotopy equivalences 
\begin{equation}\label{Atiyah}
B\mathcal{G}^f(X)\simeq{\rm{Map}}^f(X,BG),
\end{equation}
\begin{equation}\label{Atiyahun}
B\mathcal{G}^f_*(X)\simeq{\rm{Map}}^f_*(X,BG).
\end{equation}
From  \eqref{Atiyah} and \eqref{Atiyahun} we obtain the following equivalences 
\begin{equation}\label{Atiyah22}
\mathcal{G}^f(X)\simeq\Omega{\rm{Map}}^f(X,BG),
\end{equation}
\begin{equation}\label{Atiyahunn}
\mathcal{G}^f_*(X)\simeq\Omega{\rm{Map}}^f_*(X,BG).
\end{equation}
In what follows we use equivalences \eqref{Atiyah}-\eqref{Atiyahunn} to get homotopy decompositions of the gauge groups. The following lemma will be needed in order to identify certain spaces that appear in the homotopy decompositions of the gauge groups over $M$. 
\begin{lemma}\label{l:diagram}
There exists a homotopy commutative diagram of cofibrations
\begin{equation}\label{d:deltadiag}
\xymatrix{
{*}\ar[r]\ar[d]&S^{q+n}\ar@{=}[r]\ar@{^{(}->}[d]&S^{q+n}\ar[d]^{}\ar[r]&{*}\ar[d]\\
M\ar@{=}[d]\ar[r]^-{}&\bigvee_{i=1}^r S^n_i\vee S^{q+n}\ar[r]^-{}\ar[d]^{{pinch}}&\bigvee_{i=1}^rS^{q+1}_i\ar[r]^{}\ar[d]^{}&\Sigma M\ar@{=}[d]\\
M\ar[r]^-{p}&\bigvee_{i=1}^r S^n_i\ar[r]^-{}&C\ar[r]^-{}&\Sigma M
}
\end{equation}
which defines the space $C$. Furthermore
there is a homotopy equivalence
$$C \xrightarrow\simeq\bigvee_{i=1}^{r-\bar t}S^{q+1}_{i}\vee{ \Sigma Y_{\mathcal F}}.$$
\end{lemma}
\begin{proof} 
By Lemma \ref{quotient}, the inclusion of the $q$-skeleton into $M$ induces the following homotopy commutative diagram
\begin{equation}\label{csq}
\xymatrix{
\bigvee_{i=1}^rS^{q}\ar[r]^-{i}&M_{}\ar[r]^-{}\ar[d]_-{p}&\bigvee_{i=1}^rS^n_i\vee S^{n+q}\ar[ld]^-{pinch}\\
&\bigvee_{i=1}^rS^n_i&
}
\end{equation} 
We can extend \eqref{csq} to generate a homotopy commutative diagram as the one stated in the lemma, where each column and row is a cofibration sequence, and defining in this way the space $C$. Let $s:\bigvee_{i=1}^rS^n_i\to M_{}$ be a right homotopy inverse of the map $p$. Then $\Sigma p\circ\Sigma s\simeq Id$ and therefore the map $$\psi:\bigvee_{i=1}^rS^{n+1}_i\vee C\xrightarrow{\Sigma s+b} \Sigma M,$$ where $b$ is the connecting map of the cofibration induced by the projection $\pi$, is a homotopy equivalence. By Proposition \ref{p:Msuspension} there is a homotopy equivalence
\begin{equation}\label{eq:sus}
\theta:\Sigma M\xrightarrow\simeq \bigvee_{i=1}^rS^{n+1}_i\vee \bigvee_{i=1}^{r-\bar t}S^{q+1}_{i}\vee \Sigma Y_{\mathcal F},
\end{equation}
where $\bar t=min\{r,rk (N_{\mathcal F})\}$  and $\Sigma Y_{\mathcal F}$ is the cofibre of the map $\bigvee_{i=1}^d S^{n+1}_i\vee \bigvee_{j=1}^{r-\bar t}S^{q+1}_j\hookrightarrow \Sigma M$. Let $h=pinch\circ\Sigma\varphi$. The suspension of the attaching map generates a homotopy commutative diagram of cofibrations
\begin{equation}
\xymatrix{
{*}\ar[r]\ar[d]&\bigvee_{i=1}^rS^{n+1}_i\ar@{=}[r]\ar[d]&\bigvee_{i=1}^rS^{n+1}_i\ar[d]^-{\Sigma s}\\
S^{n+q}\ar[r]^-{\Sigma\varphi}\ar@{=}[d]&\bigvee_{i=1}^r(S^{n+1}_i\vee S^{q+1}_i)\ar[r]\ar[d]^-{pinch}&\Sigma M \ar[d]^-{c}\\
S^{n+q}\ar[r]^-{h_{}}&\bigvee_{i=1}^rS^{q+1}_i\ar[r]^-{}&\bigvee_{i=1}^{r-\bar t}S^{q+1}_{i}\vee\Sigma Y_{\mathcal F}
}
\end{equation}
which defines the map $c$.
The homotopy commutative diagram of cofibrations
\begin{equation}
\xymatrix{
{*}\ar[r]\ar[d]&C\ar@{=}[r]\ar[d]^-{}&C\ar[d]^-{}\\
\bigvee_{i=1}^rS^{n+1}_i\ar[r]^-{\Sigma s}\ar@{=}[d]&\Sigma M\ar[r]^-c\ar[d]^-{\Sigma p}&\bigvee_{i=1}^{r-\bar t}S^{q+1}_{j}\vee\Sigma Y_{\mathcal F} \ar[d]^-{}\\
\bigvee_{i=1}^r{ S_i^{n+1}}\ar[r]^-{\simeq}&\bigvee_{i=1}^r{ S_i^{n+1}}\ar[r]^-{}&{*}
}
\end{equation}
shows that there is a homotopy equivalence $C\xrightarrow\simeq\bigvee_{i=1}^{r-\bar t}S^{q+1}_{j} \vee\Sigma Y_{\mathcal F}$.
 \end{proof}

Let $M=\sharp_{i=1}^dM_i$, $d\geq 2$, be a manifold  where each summand $M_i$ is an $S^q$-bundle over $S^n$, $n=2k$, $q=2k'-1$, $2\leq k'\leq k$, with $\chi(M_i)=i_*\xi_i$, where $\xi_i\in\pi_{n-1}(SO(q))$. Suppose $G$ is a Lie group satisfying one of the conditions of Corollary \ref{cor:pringb}. By Propositions \ref{p:pb_gen} and \ref{classext} the map $p: M \to \bigvee_{i=1}^nS^n_i$ induces a map in path components $$p^*:\pi_0({\rm{Map}}(\bigvee_{i=1}^r S^n_i,BG)=\oplus_{i=1}^r\pi_{n-1}(G)\to \pi_0{\rm{Map}}(M,BG)=[M,BG]$$ which is a bijection. Also by Lemma \ref{l:diagram}, the map $p$ induces a homotopy fibration sequence
 \begin{equation}\label{f:indp}
{\rm Map}_*(\bigvee_{i=1}^{r-\bar t}S^{q}_{i} \vee\Sigma Y_{\mathcal F} ,BG)\xrightarrow{q^*}{\rm Map}_*(\bigvee_{i=1}^rS^n_i,BG)\xrightarrow{p^*}{\rm Map}_*(M,BG).
\end{equation}
Restricting \eqref{f:indp} to the connected component indexed by $K=(k_1,\dots,k_r)\in\mathbb Z^r$ we obtain the following fibration sequence
  \begin{equation}\label{f:indpK}
F^K\xrightarrow{q^*}\prod_{i=1}^r\Omega^{n-1}_{k_i}G\xrightarrow{p^*}{\rm Map}_*^K(M,BG),
\end{equation}
where we have identified $\prod_{i=1}^r\Omega^{n-1}_{k_i}G$ with ${\rm Map}^K_*(\bigvee_{i=1}^rS^n_i,BG)$.
\begin{lemma}\label{hofibre}
 Let $F^K$ be the homotopy fibre of 
$$p^*:\prod_{i=1}^r\Omega_{k_i}^{n-1}G \rightarrow {\rm{Map}}^K_*(M,BG),$$ $K=(k_1,\dots,k_r)\in\mathbb Z^r$. 
There are homotopy equivalences $$F^K\simeq\left(\prod_{i=1}^{r-\bar t}\Omega^{q}G\right)\times{\rm{Map}}_*(Y_{\mathcal F},G).$$
\end{lemma}
\begin{proof}
The homotopy commutative diagram of cofibrations of Lemma \ref{l:diagram} induces the following diagram where rows and columns are fibrations 
\begin{equation}\label{d:fib1}
\xymatrix{
\left(\prod_{i=1}^{r-\bar t}\Omega^{q+1}BG\right)\times{\rm{Map}}_*(\Sigma Y_{\mathcal F},BG)\ar[r] \ar[d] & \prod_{i=1}^r\Omega^{q+1}BG \ar[r]^-{\gamma^*} \ar[d]_-{(0+\gamma)^*} & \Omega^{n+q}BG \ar@{=}[d] \\
\left(\prod_{i=1}^r\Omega^nBG\right) \ar[r]^-{i^*_i} \ar[d]_-{p^*} & \left(\prod_{i=1}^r\Omega^nBG\right)\times\Omega^{n+q}BG\ar[r]^-{p_2^*} \ar[d]_-{q^*} & \Omega^{n+q}BG \ar[d] \\
{\rm Map}_*(M,BG) \ar@{=}[r] & {\rm Map}_*(M,BG) \ar[r] & \ast}
\end{equation}
We have done the following identifications $${\rm{Map}}_*(\bigvee_{i=1}^rS^\ell_i,BG)\simeq\prod_{i=1}^r\Omega^\ell BG,$$
$${\rm{Map}}_*(\bigvee_{i=1}^rS^n_i\vee S^{n+q},BG)\simeq\left(\prod_{i=1}^r\Omega^nBG\right)\times\Omega^{n+q}BG,$$ 
$${\rm Map}_*(\bigvee_{i=1}^{r-\bar t}S^{q+1}\vee\Sigma Y_{\mathcal F})\simeq\prod_{i=1}^{r-\bar t}\Omega^{q+1}BG\times{\rm{Map}}_*(\Sigma Y_{\mathcal F},BG).$$ The map $p_2^*$
 is the projection and the map $i^*_i$ is the inclusion. Observe  that $q^*$ induces a bijection in path components since $\pi_{n+q-1}(G)=0$.  Recall that there exist homotopy equivalences between the path components $\Psi_K:\prod_{i=1}^r\Omega^n_0BG\to\prod_{i=1}^r\Omega^n_{k_i}BG$ for all $K=(k_1,\dots,k_r)\in \mathbb Z^r$. These equivalences are defined by \eqref{eq:psii}
Thus the following diagram
\begin{equation}\label{d:fib2}
\xymatrix{
\left(\prod_{i=1}^{r-\bar t}\Omega^{q+1}BG\right)\times{\rm{Map}}_*(\Sigma Y_{\mathcal F},BG)\ar[r] \ar[d] & \prod_{i=1}^r\Omega^{q+1}BG \ar[r]^-{\gamma^*} \ar[d]_-{ (0+\gamma)^*} &  \Omega^{n+q}BG \ar@{=}[d] \\
\prod_{i=1}^r\Omega^n_0BG \ar[r]^-{i^*} \ar[d]_-{\theta_K} &\left(\prod_{i=1}^r\Omega^n_{0}BG\right)\times\Omega^{n+q}BG \ar[r]^-{{p_2}^*} \ar[d]_-{\theta_K\times\mathbbm 1} &\Omega^{n+q}BG \ar@{=}[d] \\
\prod_{i=1}^r\Omega^n_{k_i}BG \ar[r]^-{i^*}  & \left(\prod_{i=1}^r\Omega^n_{k_i}BG\right) \times\Omega^{2n-1}BG \ar[r]^-{{p_2}^*}
&\Omega^{n+q}BG 
}
\end{equation}
homotopy commutates. The homotopy commutativity of \eqref{d:fib2} implies that restricting $q^*$ in \eqref{d:fib1} to the $K$-th component induces a homotopy commutative diagram 
\begin{equation}
\xymatrix{
\left(\prod_{i=1}^{r-\bar t}\Omega^{q}BG\right)\times{\rm{Map}}_*(\Sigma Y_{\mathcal F},BG)\ar[r] \ar[d] & \prod_{i=1}^r\Omega^{q+1}BG \ar[r]^-{\gamma^*} \ar[d]_-{(\theta_K\times\mathbbm1) \circ (0+\gamma)^*} &  \Omega^{n+q}BG \ar@{=}[d] \\
\prod_{i=1}^r\Omega^n_{k_i}BG \ar[r]^-{i^*} \ar[d]_-{p^*} & \left(\prod_{i=1}^r\Omega^n_{k_i}BG\right) \times\Omega^{n+q}BG \ar[r]^-{{p_2}^*} \ar[d]_-{q^*} &\Omega^{n+q}BG \ar[d] \\
{\rm Map}_*^K(M,BG) \ar@{=}[r] & {\rm Map}_*^K(M,BG) \ar[r]& \ast
}
\end{equation}
Hence, for all $K\in\mathbb Z^r$,  there are homotopy equivalences
$$F^K \simeq \left(\prod_{i=1}^{r-\bar t}\Omega^{q+1}BG\right)\times{\rm{Map}}_*(\Sigma Y_{\mathcal F},BG)\simeq\left(\prod_{i=1}^{r-\bar t}\Omega^{q}G\right)\times{\rm{Map}}_*(Y_{\mathcal F},G)$$
as required.
\end{proof}

The existence of homotopy equivalences among the connected components of ${\rm{Map}}_*(M,BG)$ under our assumptions on $G$ and $M$ is an open problem. Yet we can give a result on the homotopy types of their loop spaces and, therefore, on the pointed gauge groups over manifolds $M$. 

\begin{theorem}\label{t:PGG}
Let $M=\sharp_{i=1}^rM_i$ be a manifold where each factor $M_i$ is a $S^q$-bundles over $S^n$ with cross sections. Suppose that one of the following holds:
\begin{enumerate}
\item $G=SU(m)$, $n=2k$, $q=2k'-1$, $2 \leq k' \leq k$;
\item $G=Sp(m)$, $n=4k$, $q=4k'-1$, $1 \leq k' \leq k$.
\end{enumerate}

Then for all $K\in \mathbb{Z}^r$
there is a homotopy equivalence $$\mathcal{G}^K_*(M)\simeq\prod_{i=1}^r\Omega^4 G\times\left(\prod_{i=1}^{r-\bar t}\Omega^{q}G\right)\times{\rm{Map}}_*(Y_{\mathcal F},G).$$  

Moreover, if $k=k'=2$  and $G$ is a simply connected simple compact Lie group then the equivalence holds whenever $\gcd (|\pi_6(G))|,\xi^1,\dots,\xi^r)=1$, where $\xi^i\in \pi_3(SO(3))$, $1\leq i\leq r$, is the map that classifies the manifold $M_i$.
\end{theorem}

 \begin{proof}
Let $\mathcal{G}^K_*(M_{})$ be the pointed  gauge group of the principal $G$-bundle over $M$ classified by $K\in\mathbb{Z}^r.$ 
By Lemma \ref{hofibre} the there is a homotopy fibration sequence induced by $p^*_K$
\begin{equation}\label{eq:fibpcom}
\left(\prod_{i=1}^{r-\bar t}\Omega^{q}G\right)\times{\rm{Map}}_*(Y_{\mathcal F},G)\to {\rm{Map}}_*^K(\bigvee_{i=1}^rS^n_i,BG){\xrightarrow{p_K^*}}{\rm{Map}}_*^K(M,BG).
\end{equation}
Consider the following fibration sequence
\begin{equation}\label{f:ls3inc}
 \xymatrix{
\Omega {\rm{Map}}_*^K(\bigvee_{i=1}^rS^n_i,BG)\ar[r]^-{\Omega p_K^{*}}&\Omega{\rm{Map}}_*^K(M_{},BG)\ar[r]&\left(\prod_{i=1}^{r-\bar t}\Omega^{q}G\right)\times{\rm{Map}}_*(Y_{\mathcal F},G).
}\end{equation}
For each $1\leq i\leq r $ let $s_i:M\to S^n_i$ be a cross section of the bundle $M_i\to S^n$. The  map $p:M\xrightarrow{pinch}\bigvee_{i=1}^rM_i\xrightarrow{\bigvee_{i=1}^r\pi_i}\bigvee_{i=1}^rS^n_i$ has a right homotopy inverse, namely the composite 
$$d:\bigvee_{i=1}^r S^n_i\hookrightarrow \bigvee^r_{i=1} S_i^q\vee\bigvee^r_{i=1} S_i^n\xrightarrow{g} M$$
where $g:\bigvee^r_{i=1} S_i^q\vee\bigvee^r_{i=1} S_i^n  \rightarrow  M$ be defined so that $S^q_i\hookrightarrow \bigvee^r_{i=1} S_i^q\vee\bigvee^r_{i=1} S_i^n  \xrightarrow g M$ is the inclusion of $S^q_i$ and $S^n_i\hookrightarrow \bigvee^r_{i=1} S_i^q\vee\bigvee^r_{i=1} S_i^n  \xrightarrow g M$ is the map $s_i$. Thus the diagram
\begin{equation}\label{d:section}
\xymatrix{
\bigvee_{i=1}^rS^n_i\ar[r]^-d\ar@{=}[rd]&M_{}\ar[d]^-p\\
&\bigvee_{i=1}^rS^n_i
}
\end{equation}
homotopy commutes. It follows that there is a homotopy commutative diagram
\begin{equation}\label{d:indsection}
\xymatrix{
{\rm{Map}}_*(\bigvee_{i=1}^rS^n_i,BG)\ar@{=}[rd]&{\rm{Map}}_*(M,BG)\ar_-{d^{*}}[l]\\
&{\rm{Map}}_*(\bigvee_{i=1}^rS^n_i,BG)\ar[u]_-{p^*},
}
\end{equation}
and indeed, we obtain a similar diagram for the restriction to the $K$-th component. It follows that the $(\Omega d_k)^*\circ(\Omega p_k)^*\simeq Id$ which impliest that the homotopy fibration \eqref{f:ls3inc} splits and there is a homotopy equivalence
\begin{equation}\label{e:decom_pgg1}
\Omega {\rm{Map}}^K_*(M_{},BG)\simeq\Omega{\rm{Map}}_*^K(\bigvee_{i=1}^rS^4_i,BG)\times\left(\prod_{i=1}^{r-\bar t}\Omega^{q}G\right)\times{\rm{Map}}_*(Y_{\mathcal F},G).
\end{equation}
If $n=4$ and  $\Omega {\rm{Map}}_*^K(\bigvee_{i=1}^rS^4_i,BG)\simeq\prod_{i=1}^r\Omega^4 G$ and $\mathcal{G}_{*}^k(M)\simeq \Omega{\rm{Map}}_*^K(M,BG)$. Using this identifications along with the homotopy equivalence \eqref{e:decom_pgg1} we obtain $$\mathcal{G}^K_*(M)\simeq\prod_{i=1}^r\Omega^4 G\times\left(\prod_{i=1}^{r-\bar t}\Omega^{q}G\right)\times{\rm{Map}}_*(Y_{\mathcal F},G).$$  If $n=4$, $G$ is a simply connected simple compact Lie group and $\gcd (|\pi_6(G))|,\xi^1,\dots,\xi^r)=1$, where $\xi^i\in \pi_3(SO(3))$ is the map that classifies the manifold $M_i$, then the gauge groups over $M$ are also classified by $K\in\mathbb Z^r$. Arguing in a similar manner as for the previous case, we obtain a homotopy equivalence as in \eqref{e:decom_pgg1}.\end{proof}

The result of Theorem \ref{t:PGG} simplifies the computations of the homotopy groups of the pointed gauge groups.

\begin{cor}\label{c:pgghg}
For all $K\in\mathbb Z^r$ and for all $j\geq0$ there are isomorphisms
$$\pi_j(\mathcal G^K_*(M))\cong\bigoplus_{i=1}^r \pi_{j+n}(G)\oplus\bigoplus_{i=1}^{r-\bar t}\pi_{j+q}(G)\oplus\pi_j({\rm{Map}}_*(Y_{\mathcal F},G)).$$
\qed
\end{cor}

Consider the evaluation fibration 
\begin{equation}
{\rm{Map}}_{K}^*{(M,BG)}\rightarrow{\rm{Map}}_{K}(M,BG)\overset{ev_K}{\longrightarrow}BG,
\end{equation}

The evaluation map $ev$ is natural, and by Proposition \ref{classext}, the map $p:M\rightarrow \bigvee^d_{i=1} S^4$ makes the homotopy fibration diagram
\begin{equation}\label{d:eval}
\xymatrix@=0.3in{ 
G\ar[r]^-{\phi_{K}}\ar@{=}[d]&\mapp^K({\bigvee_{i=1}^r}{S^n_i},BG)\ar[r]\ar[d]^-{p^{*}}&
\map^{K}({\bigvee}_{i=1}^rS^n_i,BG)\ar[r]^-{ev_K}\ar[d]^-{p^{*}}&BG\ar@{=}[d]\\
G\ar[r]^-{\partial_{K}}&\mapp^K(M,BG)\ar[r]&{\rm{Map}}^{K}(M,BG)
\ar[r]^-{ev_K}
&BG
}
\end{equation}
homotopy commute. Now we are ready to give a proof of Theorem \ref{t:UGG}  .

\begin{proof}[Proof of Theorem \ref{t:UGG}]
The left square in \eqref{d:eval} induces a homotopy commutative diagram 
\begin{equation}\label{d:ugg}
\xymatrix@=0.3in{
&{\rm{Map}}_*^{K}(\Sigma M,BG)\ar@{=}[r]\ar[d]&{\rm{Map}}_*^{K}(\Sigma M,BG)\ar[d]^{\delta^*}\\
\mathcal{G}^{K}(\bigvee_{i=1}^rS^n_i )\ar[r]\ar@{=}[d]&\mathcal G_{K}(M)\ar[r]^-{\zeta}\ar[d]&
\left(\prod_{i=1}^{r-\bar t}\Omega^{q}G\right)\times{\rm{Map}}_*(Y_{\mathcal F},G)\ar[d]^{g^{*}}\\
\mathcal{G}^{K}(\bigvee_{i=1}^rS^n_i)\ar[r]&G\ar[r]^-{\phi_{K}}  \ar[d]^-{\partial_K} &{\rm{Map}}^{K}_{*}(\bigvee_{i=1}^rS^n_i,BG) \ar[d]^-{p^*} \\
& {\rm{Map}}_{*}^{K}(M,BG) \ar@{=}[r] & {\rm{Map}}_{*}^{K}(M,BG)
}
\end{equation}
which defines the map $\zeta$. Since the map $\delta^*$ has a right homotopy inverse, so does the map $\zeta$. We use the group structure on $\mathcal{G}_k(M)$  to obtain a homotopy equivalence
\begin{equation}\label{e:partialeqgg}
\mathcal{G}_k(\bigvee_{i=1}^rS^n_i)\times\left(\prod_{i=1}^{r-\bar t}\Omega^{q}G\right)\times{\rm{Map}}_*(Y_{\mathcal F},G)\rightarrow \mathcal{G}_k(M)\times \mathcal{G}_k(M)\rightarrow \mathcal{G}_k(M).
\end{equation}
The homotopy splitting \eqref{decomp} follows from the homotopy equivalence  \eqref{e:partialeqgg} and Proposition \ref{p:wedgegg}. For the second part of the theorem, if $n = 4$, $q = 3$, using Proposition \ref{p:pb_gen} and Therorem \ref{t:PGG} we obtain a homotopy commutative diagram as in \eqref{d:ugg} whenever $(|\pi_6(G)|, \xi^1,\dots,\xi^r)=1$, $\xi^i\in\pi_3(SO(3))\cong\mathbb Z$ for any simply connected simple compact Lie group $G$. In this case $\bar t=1$ and $S^6\xrightarrow{\alpha}S^3$ is a unit in $\pi_6(S^3)\cong\mathbb Z_{12}$. Similar arguments show there is also a homotopy equivalence as in \ref{e:partialeqgg}. Proposition  \ref{p:wedgegg} completes the proof.

\end{proof}

\begin{remark}
Given $K\in\mathbb Z^r$, the evaluation map induces the following exact sequences \begin{equation}\label{e:hgev}
\cdots\rightarrow\pi_j(\mathcal G^K_*(M))\xrightarrow{i^*}\pi_j(\mathcal G^K(M))\xrightarrow{ev^*}\pi_j(G)\rightarrow\cdots
\end{equation}
For any simply connected simple compact Lie group we have $\pi_n(G)=0$ for $n\leq 2$. This implies that if $n\leq2$ then $i^{*}$ is an isomorphism. We can use these isomorphisms and Corollary 5.4 to compute the path components of unpointed gauge groups $\mathcal G^K(M)$. For example, let $n=4$, $q=3$ and $M=\sharp_{i=1}^{r} M_{\xi^i}$ with characteristic elements $\xi^i\in\pi_3(SO(3))\cong Z$. If $\xi^1\equiv 1\pmod {12}$ and $\xi^i\equiv0\pmod {12}$, for all $2\leq i\leq r$, then a similar argument to the one given in the proof of Lemma \ref{l:coker}  shows that $\pi_0(\map_*(Y_{\mathcal F},BG))=[Y_{\mathcal F},BG]=0$. We use information of the homotopy groups of Lie groups as given in \cite{Jam2} to obtain:
\begin{equation*}
\pi_0(\mathcal G^k(M))=\begin{cases}
\bigoplus_{i=1}^r\mathbb Z_2\oplus\bigoplus_{i=1}^{r-1}\mathbb Z& G=SU(2),Sp(n)(n\geq2),Spin(5)\\
\bigoplus_{i=1}^{r-1}\mathbb Z & G=SU(n)(n\geq 3), Spin(m)(m \geq 6)\\
\bigoplus_{i=1}^{r-1}\mathbb Z& G=G_2,F_4,E_6,E_7,E_8.
\end{cases}
\end{equation*}
\end{remark}
 The connecting map 
$$G\xrightarrow{\bar\delta_1}\Omega^{n-1}_1G\longrightarrow{\rm{Map}}^1(S^n,BG)\xrightarrow{ev} BG$$  
    determines upper bounds in the number of homotopy types of $\mathcal{G}^K(M_{})$. For example, it is known that the order of the connecting map  on $SU(2)$-gauge groups over $S^4$ is 12. Using this result along with some exact sequences of homotopy groups, we obtain a classification result for manifolds $M=\sharp_{i=1}^r M_{\xi^i}$, where each summand $M_{\xi^i}$ is an $S^3$-bundle over $S^4$, $\xi^i\in\pi_3(SO(3))\cong\mathbb Z$ and  $(12,\xi^1,\dots,\xi^r)=1$. These kinds of manifolds include, for instance, the connected sums $M=S^3\tilde\times_{\xi^1}S^4\,\sharp\,(\sharp_{i=2}^r S^3\times S^4)$, where $S^3\tilde\times_{\xi^1}S^4$ is a twisted product and the class of $\xi^1$ mod 12 is a unit. 

\begin{cor}\label{t:cor}
Let $M=\sharp_{i=1}^r M_{\xi^i}$ be a manifold where each factor $M_i$ is an $S^{3}$-bundle over $S^{4}$ with a cross section, $G = SU(2)\cong S^3$ and $\gcd(12,\xi^1,\dots,\xi^r)=1$. Given $K,K'\in\mathbb Z^r$ with $K=(k_1,\dots,k_r)$ and $K'=(k'_1,\dots,k'_r)$, there is a homotopy equivalence $\mathcal{G}^K(M)\simeq\mathcal{G}^{K'}(M)$ if and only if $\gcd(12,k_1,\dots,k_r)=\gcd(12,k'_1,\dots,k'_r).$
\end{cor}

\begin{proof}

By Theorem  \ref{t:UGG},  for all $\tilde K=(\tilde k_1,\dots,\tilde k_r)\in \mathbb Z^r$, there are homotopy equivalences
\begin{equation}\label{e:decomggM}
 \mathcal G^{\tilde K}(M)\simeq \mathcal G^{\lambda(\tilde K)}(S^4)\times \prod_{i=1}^{r-1}\Omega^{4}S^3\times \prod_{j=1}^{r-\bar t}\Omega^{3}S^3\times \mathrm{Map}_*(Y_{\mathcal F},S^3),
 \end{equation}
where $\lambda(\tilde K)=gcd(o(\delta_1),\tilde k_1,\dots,\tilde k_r)$, $r,\bar t\in\mathbb N$ that does not depend on $\tilde K$ and $\delta_1$ is the connecting map of 
the fibration 
\begin{equation}
S^3\xrightarrow{\delta_1}\Omega^{3}_1S^3\longrightarrow{\rm{Map}}^1(S^4,BS^3)\xrightarrow{ev} BS^3.
\end{equation}
According to \cite{AK}  $o(\delta_1)=12$. Thus given $K,K'\in\mathbb Z^r$, if $\gcd(12,k_1,\dots,k_r)=\gcd(12,k_1',\dots,k_r')$ then $\lambda(K)=\lambda(K')$ and from the homotopy equivalence \eqref{e:decomggM} it follows that $\mathcal G^K(M)\simeq \mathcal G^{K'}(M).$

Now suppose that $\mathcal G^K(M)\simeq\mathcal G^{K'}(M)$. Then there are isomorphisms  
\begin{equation}\label{e:hg1}
\pi_n(\mathcal G^K(M))\cong\pi_n(\mathcal G^{K'}(M))
\end{equation}
for all $n\geq0$.  The homotopy equivalence \eqref{e:decomggM}  induces an isomorphism
\begin{equation}\label{e:hg2}
\pi_2(\mathcal G^K(M))\cong\pi_2(\mathcal G^{\lambda(K)}(S^4)) \oplus\bigoplus_{i=1}^{r-1}\pi_{6}(S^3)\oplus\bigoplus_{i=1}^{r-\bar t}\pi_5(S^3)\oplus\pi_2({\rm{Map}}^*(Y_{\mathcal F},S^3))
\end{equation}
 Notice that $[\Sigma^2 Y_{\mathcal F},S^3]=\pi_2({\rm{Map}}^*(Y_{\mathcal F},S^3))$ and there exists an exact sequence 
\begin{equation}
\pi_9(S^3)\to[\Sigma^2 Y_{\mathcal F},S^3]\to\pi_5(S^3).
\end{equation}
 Since $\pi_n(S^3)$ is finite for $n>3$ 
 then $[\Sigma^2 Y_{\mathcal F},S^3]$ is finite. 
 Since  $o(\delta_1)=12$, the exact sequence 
 
 \begin{equation*}
\xymatrix{
\pi_4(BS^3)\ar[d]^-{\cong}\ar[rr]^-{(\delta_{\lambda(K)})^*}&&\pi_3(\Omega^4 BS^3)\ar[d]^-{\cong}\ar[r]&\pi_3(B\mathcal G^{\lambda(K)}(S^4))\ar[r]&0\\
\pi_{3}(S^3)\ar[rr]^-{(\lambda(K)\circ\delta_1)^*}&&\pi_6(S^3)=\mathbb Z_{12}
}
\end{equation*}
 shows that $|\pi_2(\mathcal G^{\lambda(K)}(S^4))|=\lambda(K)$. Observe that the group $$A=\bigoplus_{i=1}^{r-1}\pi_{6}(S^3)\oplus\bigoplus_{i=1}^{r-\bar t}\pi_5(S^3)\oplus\pi_2({\rm{Map}}^*(Y_{\mathcal F},S^3))$$ is finite. It follows from  equations \eqref{e:hg1} and \eqref{e:hg2} that 
 $$\lambda(K)|A|=|\pi_2(\mathcal G^{\lambda(K)}(M))|=|\pi_2(\mathcal G^{\lambda(K')}(M))|=\lambda(K')|A|.$$
Therefore $\gcd(12,k_1,\dots,k_r) =\lambda(K)=\lambda(K')=\gcd(12,k_1',\dots,k_r').$
\end{proof}

\begin{ack}
I would like to thank Stephen Theriault and Shizuo Kaji for invaluable conversations and feedback on  this manuscript. I also want to thank Jacek Brodzki and the JTD group at the University of Southampton for giving me excellent conditions to work. This research was partially supported by the Mexican National Council for Science and Technology (CONACyT) through the scholarship 313812 and by the Grace Chisholm Fellowship of the London Mathematical Society.
\end{ack}

\bibliographystyle{amsplain}
\bibliography{bibliography}

\providecommand{\bysame}{\leavevmode\hbox to3em{\hrulefill}\thinspace}
\providecommand{\MR}{\relax\ifhmode\unskip\space\fi MR }
\providecommand{\MRhref}[2]{%
  \href{http://www.ams.org/mathscinet-getitem?mr=#1}{#2}
}
\providecommand{\href}[2]{#2}
\begin{thebibliography}{10}

\bibitem{Ark}
M.~Arkowitz, \emph{The group of self-homotopy equivalences - a survey},
  Springer, 1990.

\bibitem{AB}
M.~F. Atiyah and R.~Bott, \emph{The {Y}ang-{M}ills equations over {R}iemann
  surfaces}, Philos. Trans. R. Soc. Lond. A \textbf{308} (1983), no.~1505,
  523--615.

\bibitem{Spf}
M.~H.~A. Claudio and M.~Spreafico, \emph{Homotopy type of gauge groups of
  quaternionic line bundles over spheres}, Topol. Appl. \textbf{156} (2009),
  no.~3, 643--651.

\bibitem{CS}
M.~C. Crabb and W.~A. Sutherland, \emph{Counting homotopy types of gauge
  groups}, Proc. Lond. Math. Soc. \textbf{81} (2000), no.~3, 747--768.

\bibitem{DT}
S.~K. Donaldson and R.~P. Thomas, \emph{Gauge theory in higher dimensions}, The
  Geometric Universe: Science, Geometry, and the Work of Roger Penrose, Oxford
  University Press, 1998, pp.~31--47.

\bibitem{Got}
D.~H. Gottlieb, \emph{On fibre spaces and the evaluation map}, Ann. Math.
  \textbf{87} (1968), no.~1, 42--55.

\bibitem{HKK}
H.~Hamanaka, S.~Kaji, and A.~Kono, \emph{Samelson products in {Sp}(2)}, Topol.
  Appl. \textbf{155} (2008), no.~11, 1207--1212.

\bibitem{RH1}
R.~Huang, \emph{Homotopy types of gauge groups over high dimensional
  manifolds}, preprint, available at
  \href{https://arxiv.org/abs/1805.04879}{arXiv:1805.04879}, 2018.

\bibitem{Ish}
H.~Ishimoto, \emph{Homotopy classification of connected sums of sphere bundles
  over spheres, {I}}, Nagoya Math. J. \textbf{83} (1981), 15--36.

\bibitem{Jam2}
I.~M. James (ed.), \emph{Handbook of algebraic topology}, Elsevier, 1995.

\bibitem{JW1}
I.~M. James and J.~H.~C. Whitehead, \emph{The homotopy theory of sphere bundles
  over spheres ({I})}, Proc. Lond. Math. Soc. \textbf{3} (1954), no.~4,
  196--218.

\bibitem{KKT}
D.~Kishimoto, A.~Kono, and M.~Tsutaya, \emph{On $p$-local homotopy types of
  gauge groups}, Proc. R. Soc. Edinb. A \textbf{144} (2014), no.~1, 149--160.

\bibitem{AK}
A.~Kono, \emph{A note on the homotopy type of certain gauge groups}, Proc. R.
  Soc. Edinb. A \textbf{117} (1991), no.~3--4, 295--297.

\bibitem{Lng}
G.~E. Lang, \emph{The evaluation map and ${EHP}$ sequences}, Pacific J. Math.
  \textbf{44} (1973), no.~1, 201--210.

\bibitem{IAM2}
I.~Membrillo-Solis, \emph{Homotopy theory of gauge groups over certain
  7-manifolds}, Ph.D. thesis, University of Southampton, 2018.

\bibitem{IMS}
\bysame, \emph{Homotopy types of gauge groups related to ${S}^3$-bundles over
  ${S}^4$}, Topol. Appl. \textbf{255} (2019), 56--85.

\bibitem{MimJam}
M.~Mimura, \emph{{H}omotopy theory of {L}ie groups}, Handbook of Algebraic
  Topology (I.~M. James, ed.), Elsevier, 1995.

\bibitem{Rut}
J.~W. Rutter, \emph{A homotopy classification of maps into an induced fibre
  space}, Topology \textbf{6} (1967), no.~3, 379--403.

\bibitem{Rut2}
\bysame, \emph{Spaces of homotopy self-equivalences - a survey}, Lecture Notes
  in Mathematics, vol. 1662, Springer, 2006.

\bibitem{Saw}
N.~Sawashita, \emph{On the self-equivalences of {H}-spaces}, J. Math. Univ.
  Tokushima \textbf{10} (1976), 17--33.

\bibitem{Th2}
S.~Theriault, \emph{The homotopy types of ${Sp}(2)$-gauge groups}, Kyoto J.
  Math. \textbf{50} (2010), no.~3, 591--605.

\bibitem{Tod}
H.~Toda, \emph{Composition methods in homotopy groups of spheres}, Annals of
  Mathematics Studies, vol.~49, Princeton University Press, 1963.

\end{thebibliography}
\end{document}